\documentclass[12pt]{amsart}

\newtheorem{theorem}{Theorem}[section]
\newtheorem{lemma}[theorem]{Lemma}
\newtheorem{proposition}[theorem]{Proposition}
\newtheorem{corollary}[theorem]{Corollary}   
\newtheorem{definition}[theorem]{Definition}
\newtheorem{example}[theorem]{Example}

\newtheorem{remark}[theorem]{Remark}
\newtheorem{conjecture}[theorem]{Conjecture}
\newtheorem{question}[theorem]{Question}

\numberwithin{equation}{section}

\usepackage{times}
\usepackage{enumerate}
\usepackage{mathrsfs}
\usepackage{tfrupee}
\usepackage{tikz}
\usetikzlibrary{chains,fit}
\usetikzlibrary{shapes,snakes}
\usetikzlibrary{graphs}
\usepackage{graphicx,adjustbox}
\usepackage{hyperref}
\usepackage{amsmath,amssymb}
\usepackage{amscd}
\usepackage{graphicx}
\usepackage[all]{xy}

\usepackage{float}
\usepackage{caption}
\usepackage{subcaption}

\usepackage{cleveref}

\begin{document}

\title{Cohen-Macaulay Binomial edge ideals in terms of blocks with whiskers}
\author{
Kamalesh Saha \and Indranath Sengupta
}
\date{}

%Author1
\address{\small \rm  Discipline of Mathematics, IIT Gandhinagar, Palaj, Gandhinagar, 
Gujarat 382355, INDIA.}
\email{kamalesh.saha@iitgn.ac.in}

%Author2
\address{\small \rm  Discipline of Mathematics, IIT Gandhinagar, Palaj, Gandhinagar, 
Gujarat 382355, INDIA.}
\email{indranathsg@iitgn.ac.in}
\thanks{The second author is the corresponding author; supported by the 
MATRICS research grant MTR/2018/000420, sponsored by the SERB, Government of India.}

\date{}



\allowdisplaybreaks

\begin{abstract}
For a graph $G$, Bolognini et al. have shown $J_{G}$ is strongly unmixed $\Rightarrow$ $J_{G}$ is Cohen-Macaulay $\Rightarrow$ $G$ is accessible, where $J_{G}$ denotes the binomial edge ideals of $G$. Accessible and strongly unmixed properties are purely combinatorial. We give some motivations to focus only on blocks with whiskers for the characterization of all $G$ with Cohen-Macaulay $J_{G}$. We show that accessible and strongly unmixed properties of $G$ depend only on the corresponding properties of its blocks with whiskers and vice versa. Also, we give an infinite class of graphs whose binomial edge ideals are Cohen-Macaulay, and from that, we classify all $r$-regular $r$-connected graphs such that attaching some special whiskers to it, the binomial edge ideals become Cohen-Macaulay. Finally, we define a new class of graphs, called \textit{strongly $r$-cut-connected} and prove that the binomial edge ideal of any strongly $r$-cut-connected accessible graph having at most three cut vertices is Cohen-Macaulay.
\end{abstract}

\maketitle

\section{Introduction}
\medskip

Let $G$ be a simple graph (i.e., a finite undirected graph without multiple edges and loops) on the vertex set $V(G)=[n]=\{1,\ldots,n\}$. Consider the ring $R=K[x_{1},\ldots,x_{n},y_{1},\ldots,y_{n}]$, where $K$ is a field. The binomial edge ideal of $G$, denoted by $J_{G}$, is the ideal of $R$ defined as
$$J_{G}=\big<f_{ij} = x_{i}y_{j}-x_{j}y_{i}\mid \{i,j\}\in E(G)\,\, \text{with}\,\, i<j\big>.$$
The study of binomial edge ideals have been started in 2010 through the articles \cite{hhhrkara} and \cite{ohtani} independently. Concept of binomial edge ideals arises from the study of the ideal generated by $2$-minors of a $2\times n$-generic matrix.\par 

Throughout the previous eleven years, many works related to the algebraic properties and invariants of these ideals have been done (see \cite{bbs17}, \cite{crin11}, \cite{ez15}, \cite{herzogrin18}, \cite{jks20}, \cite{jksrees21}, \cite{ksm16}, \cite{mm13}, \cite{rsmkconj21}, \cite{smk18}). Generally, people try to see these things in terms of the combinatorial properties of the underlying graph. We are interested to classify those $G$ for which $J_{G}$ is Cohen-Macaulay. Although, several works have been done in this direction (see \cite{mont}, \cite{banlui}, \cite{bms_cmbip}, \cite{acc}, \cite{ehh_cmbin}, \cite{ert_licci}, \cite{hhhrkara}, \cite{s2acc}, \cite{ksara_cmunm}, \cite{cactus}, \cite{rin_smaldev}), but full characterization of Cohen-Macaulay binomial edge ideals is still widely open. \par 

To give a combinatorial characterization of Cohen-Macaulay binomial edge ideals, Bolognini et al., in \cite{acc}, have introduced two combinatorial properties of graphs: accessible (Definition \ref{defacc}) and strongly unmixed (Definition \ref{defsu}) property. Specifically, they have proved that $J_{G}$ is strongly unmixed $\Rightarrow$ $J_{G}$ is Cohen-Macaulay $\Rightarrow$ $G$ is accessible. Moreover, they conjectured \cite[Conjecture 1.1]{acc} on the equivalency of these three properties. In \cite{s2acc}, the authors showed if $R/J_{G}$ satisfies Serre's condition $S_{2}$, then $G$ is accessible. For any ideal $I\subseteq R$, it is an well known result that $R/I$ is Cohen-Macaulay if and only if $R/I$ satisfies Serre's condition $S_{r}$ for all $r\geq 1$. Therefore, combining the above results we get for any graph $G$,
{\small
$$ J_{G}\,\, \text{strongly\,\, unmixed} \Rightarrow J_{G}\,\, \text{Cohen-Macaulay} \Rightarrow R/J_{G}\,\, \text{is}\,\, S_{2} \Rightarrow G\,\, \text{accessible}.$$
}
Merging \cite[Conjecture 1.1]{acc} and \cite[Conjecture 0.1]{s2acc}, we get the following.

\begin{conjecture}\label{accsuconj}
Let $G$ be a graph. Then
{\small
$$ J_{G}\,\, \text{strongly\,\, unmixed} \Leftrightarrow J_{G}\,\, \text{Cohen-Macaulay} \Leftrightarrow R/J_{G}\,\, \text{is}\,\, S_{2} \Leftrightarrow G\,\, \text{accessible}.$$
}
\end{conjecture}

\noindent To prove the above Conjecture \ref{accsuconj}, it is enough to show that $G$ is accessible implies $J_{G}$ is strongly unmixed. \par 

A \textit{block} of a connected graph $G$ is a maximal induced subgraph of $G$ which has no cut vertex. In many papers, we have seen examples of Cohen-Macaulay $J_{G}$ are some blocks with whiskers (see \cite{bms_cmbip}, \cite{acc}, \cite{s2acc}, \cite{cactus}, \cite{rin_smaldev}). We give the motivation to study only blocks with some whiskers to characterize all Cohen-Macaulay $J_{G}$. Also, to prove the Conjecture \ref{accsuconj}, we need to focus only on blocks with whiskers. We give some classes of binomial edge ideals in support of Conjecture \ref{accsuconj}. The paper is arranged in the following manner.\par 

In Section \ref{preli}, we recall some definitions, concept, notations related to graph theory and commutative algebra. Then we mention some results from \cite{acc}, \cite{s2acc}, \cite{raufrin} which have been used frequently in our work.\par 

In Section \ref{secblocksu}, we first prove some results regarding primary decomposition and unmixedness of $J_{G}$ and accessibility of $G$ for the sake of the rest of the paper. The main results of this section are the following.

\begin{theorem}[Theorem \ref{baracc}, \ref{blocksu}, \ref{baraccG} and \ref{barsuG}]\label{blockthm}
Let $G$ be a graph. Then the following are equivalent.
\begin{enumerate}[(i)]
\item $J_{G}$ is strongly unmixed (resp. $G$ is accessible).
\item $J_{G}$ is unmixed and $J_{\overline{B}}$ is strongly unmixed (resp. $\overline{B}$ is accessible) for each block $B$ of $G$.
\end{enumerate}
\end{theorem}

\noindent The above Theorem \ref{blockthm} ensures that it is enough to study only blocks with whiskers for characterization of accessible graphs, strongly unmixed and Cohen-Macaulay binomial edge ideals. Also, due to the Theorem \ref{blockthm}, the Conjecture \ref{accsuconj} boils down to only blocks with whiskers instead of any graph $G$. At the end, in Theorem \ref{identifythm}, we settle down an open problem (see Question \ref{quesidentify}) given in \cite[Problem 7.2]{acc} for the case of accessibility of $G$ and strongly unmixedness of $J_{G}$. \par 

In Section \ref{secstarsu}, motivated from \cite[Question 5.4]{s2acc}, we started to find $r$-connected planar accessible graphs. Finally, we manage to find all $r$-regular $r$-connected non-complete blocks with whiskers, denoted by $\overline{K_{r}\star K_{r}}$, $r\geq 2$ (see \ref{secstarsu}), for which the Conjecture \ref{accsuconj} hold (see Theorem \ref{starthm}). But, among them the only planar graphs are $\overline{K_{2}\star K_{2}}$ and $\overline{K_{3}\star K_{3}}$.\par 

In Section \ref{seccutcon}, we define new classes of graphs called $r$-cut-connected and strongly $r$-cut-connected (Definition \ref{r-cutcon} and \ref{strongly-r-cutcon}). We give the following theorem and Example \ref{exallblocksu} compiling those blocks with whiskers for which the Conjecture \ref{accsuconj} hold.

\begin{theorem}[Theorem \ref{allblocksu}]
Let $G$ be a graph such that every block $B$ of $G$ satisfies any of the following conditions:\\
$(a)$ $B$ is chordal; $(b)$ $\overline{B}$ is traceable; $(c)$ $B$ is a chain of cycles (see \cite[Definition 4.2]{s2acc}; $(d)$ $B=K_{m}\star_{r} K_{n}$; $(e)$  $\overline{B}$ is strongly $3$-cut-connected containing at most $3$ cut vertices of $G$. Then the following are equivalent:
\begin{enumerate}[(i)]
\item $J_{G}$ is Cohen-Macaulay;
\item $R/J_{G}$ is $S_{2}$;
\item $G$ is accessible;
\item $J_{G}$ is unmixed and each $\overline{B}$ is accessible.
\item $J_{G}$ is strongly unmixed.
\end{enumerate}
\end{theorem}
\noindent At the end, we conclude by putting Question \ref{quescmblock} and \ref{quesidentifycm} keeping similarities with our key results.

\section{Preliminaries}\label{preli}
\medskip

In this article, we assume all graphs are simple. For a graph $G$, we denote the vertex set by $V(G)$ and edge set by $E(G)$. For a subset $W\subseteq V(G)$, the induced subgraph of $G$ on the vertex set $W$ is denoted by $G[W]$ and for $T\subseteq V(G)$, we mean by $G\setminus T$ as the graph $G[V(G)\setminus T]$. If $\{u,v\}\in E(G)$, then we say $u$ is adjacent to $v$ or vice versa. Similarly, we say $v$ is adjacent to $A$ (or $A$ is adjacent to $v$) if $v$ is adjacent to a vertex in $A$, where $A\subseteq V(G)$ and $v\in V(G)$.\par 

For a veretex $v\in V(G)$, we call $\mathcal{N}_{G}(v)=\{u\in V(G)\mid \{u,v\}\in E(G)\}$ the neighbor set of $v$ in $G$. We denote by $\mathrm{deg}_{G}(v)=\vert \mathcal{N}_{G}(v)\vert$ the \textit{degree} of a vertex $v$ in $G$. If $\mathcal{N}_{G}(v)=\{u\}$, then $\{u,v\}\in E(G)$ is called a \textit{whisker} attached to $u$. A \textit{path} from $u$ to $v$ of length $n$ in $G$ is a sequence of vertices $u=v_{0},\ldots,v_{n}=v\in V(G)$ such that $\{v_{i-1},v_{i}\}\in E(G)$ for each $1\leq i,j\leq n$ and $v_{i}\neq v_{j}$ if $i\neq j$. A \textit{chordless path} of length $n$, denoted by $P_{n}$, is a path of length $n$ without any induced cycle in it.\par 

A graph is called \textit{complete} if there is an edge between every pair of vertices and $K_{n}$ denotes the complete graph on $n$ vertices. A vertex $v\in V(G)$ is said to be a \textit{free vertex} of $G$ if the induced subgraph $G[\mathcal{N}_{G}(v)\cup \{v\}]$ is complete. A graph $G$ is called \textit{decomposable} into $G_{1}$ and $G_{2}$ if $G=G_{1}\cup G_{2}$ with $V(G_{1})\cap V(G_{2})=\{v\}$ such that $v$ is free vertex of both $G_{1}$ and $G_{2}$.\par 

A vertex $v\in V(G)$ is said to be a \textit{cut vertex} or \textit{cut point} of $G$ if removal of $v$ from $G$ increases the number of connected components. A connected graph $G$ is called \textit{$r$-connected} if removal of any set of vertices with cardinality less than $r$ keeps the graph connected. A graph $G$ is called \textit{$r$-regular} if $\mathrm{deg}_{G}(v)=r$ for all $v\in V(G)$. \par 

Let $G$ be a graph on the vertex set $V(G)=[n]$. A set $T\subseteq [n]$ is said to be a \textit{cutset} of $G$ (or we said $T$ has a cut point property) if each $t\in T$ is a cut vertex of $G\setminus (T\setminus\{t\})$. We denote by $\mathfrak{C}(G)$ the set of all cutsets of $G$. For $T\subseteq [n]$, we denote the number of connected components of the graph $G\setminus T$ by $c_{G}(T)$ (or sometimes by $c(T)$ if the graph is clearly understood from the context). Let $G_{1},\ldots,G_{c(T)}$ be the connected components of $G\setminus T$. For each $G_{i}$, we denote by $\tilde{G_{i}}$, the complete graph on the vertex set $V(G_{i})$. We set
$$ P_{T}(G)=\left\langle \bigcup_{i\in T}\lbrace x_{i},y_{i}\rbrace, J_{\tilde{G_{1}}},\ldots,J_{\tilde{G}_{c(T)}}\right\rangle.$$
Then $P_{T}(G)$ is a prime ideal and $J_{G}=\cap_{T\subseteq [n]}P_{T}(G)$. By \cite{hhhrkara}, $P_{T}(G)$ is a minimal prime ideal of $J_{G}$ if and only if $T\in\mathfrak{C}(G)$ i.e., 
$$J_{G}=\bigcap_{T\in\mathfrak{C}(G)} P_{T}(G),$$ 
is the minimal primary decomposition of $J_{G}$.
\medskip

From \cite[Lemma 3.1]{hhhrkara},  $\mathrm{height}\,P_{T}(G)=n+\vert T\vert -c(T)$. Since $\phi\in \mathfrak{C}(G)$, $J_{G}$ is unmixed (i.e., heights of all minimal primes of $J_{G}$ are same) if and only if $c(T)=\vert T\vert +c$ for any $T\in \mathfrak{C}(G)$, where $c$ denotes the number of connected components of $G$. It follows from \cite[Proposition 2.1]{raufrin}, that if $v$ is a free vertex of $G$, then $v\not\in T$ for all $T\in\mathfrak{C}(G)$.\par 

For a vertex $v$ of a graph $G$, we denote by $G_{v}$ the following graph:
$$ V(G_{v})=V(G)\,\, \text{and}\,\, E(G_{v})=E(G)\cup \{\{i,j\}\mid i,j\in \mathcal{N}_{G}(v), i\neq j\}.$$
Note that, $v$ is a free vertex of $G_{v}$.\par 

Let us recall some definitions and results from \cite{acc}, \cite{s2acc}, \cite{raufrin}.
\medskip

 \begin{definition}[\cite{acc}, Definition 2.2]\label{defacc}{\rm
Let $G$ be a graph. A cutset $T\in \mathfrak{C}(G)$ is said to be \textit{accessible} if there exists $t\in T$ such that $T\setminus\{t\}\in \mathfrak{C}(G)$. We say $\mathfrak{C}(G)$ is an \textit{accessible set system} if every non-empty cutset $T\in\mathfrak{C}(G)$ is accessible. The graph $G$ is said to be \textit{accessible} if $J_{G}$ is unmixed and $\mathfrak{C}(G)$ is an accessible set system.
}
\end{definition}

\begin{definition}[\cite{acc}, Definition 5.6]\label{defsu}{\rm
Let $G$ be a graph. We say $J_{G}$ is strongly unmixed if the connected components of $G$ are complete graphs or if $J_{G}$ is unmixed and there exists a cut vertex $v$ of $G$ for which $J_{G\setminus\{v\}}, J_{G_{v}}$ and $J_{G_{v}\setminus\{v\}}$ are strongly unmixed. Sometimes we will say $G$ is strongly unmixed instead of saying $J_{G}$ is strongly unmixed.
}
\end{definition}
The following results we have used extensively in our work.

\begin{theorem}[\cite{acc}, Theorem 4.12]\label{accthm}
Let $G$ be a connected non-complete accessible graph. Then
\begin{enumerate}[(i)]
\item every non-empty cutset of $G$ contains a cut vertex;
\item the induced subgraph on the cut vertices of $G$ is connected;
\item every vertex of $G$ is adjacent to a cut vertex.
\end{enumerate}
\end{theorem}

\begin{proposition}[\cite{acc}, Proposition 4.18]\label{accnoncut}
Let $G$ be an accessible graph and $T\in \mathfrak{C}(G)$. If $T$ contains some non-cut vertices, then $T\setminus\{t\}\in \mathfrak{C}(G)$ for some non-cut vertex $t\in T$.
\end{proposition}

\begin{proposition}[\cite{acc}, Proposition 5.2]\label{accpropunm}
Let $G$ be a connected graph such that $J_{G}$ is unmixed and $v$ be a cut vertex of $G$. If $H_{1}$ and $H_{2}$ are the connected components of $H = G\setminus\{v\}$, then the following are equivalent:
\begin{enumerate}[(i)]
\item $J_{H}$ is unmixed;
\item $T\in\mathfrak{C}(H)$ implies $\mathcal{N}_{H_{1}}(v)\not\subseteq T$ and $\mathcal{N}_{H_{2}}(v)\not\subseteq T$;
\item $\mathfrak{C}(H)=\{T\subseteq V(H)\mid T\cup\{v\}\in\mathfrak{C}(G)\}$.
\end{enumerate}
\end{proposition}

\begin{lemma}[\cite{s2acc}, Lemma 4.10]\label{freeGsu}
Let $G$ be a graph and $v$ be a free vertex of $G$. If $J_{G}$ is unmixed and $J_{G\setminus\{v\}}$ is strongly unmixed, then $J_{G}$ is strongly unmixed.
\end{lemma}

\begin{remark}{\rm
For a graph $G$ with connected components $G_{1},\ldots, G_{r}$, $J_{G}$ is unmixed (resp. Cohen-Macaulay, strongly unmixed) if and only if $J_{G_{i}}$ is unmixed (resp. Cohen-Macaulay, strongly unmixed) for each $i=1,\ldots,r$. The same holds for the accessible property of $G$. So, we will assume any graph is connected if it is not clear from the context.
}
\end{remark}

\begin{remark}\label{remglu}{\rm
Let $G=G_{1}\cup G_{2}$ with $V(G_{1})\cap V(G_{2})=\{v\}$ be a decomposable graph, where $v$ is a free vertex in both $G_{1}$ and $G_{2}$. Then
\begin{enumerate}[(i)]
\item $J_{G}$ is unmixed (resp. Cohen-Macaulay) if and only if $J_{G_{i}}$ is unmixed (resp. Cohen-Macaulay) for each $i=1,2$ (see \cite{raufrin}).

\item $J_{G}$ is accessible (resp. strongly unmixed) if and only if $J_{G_{i}}$ is accessible (resp. strongly unmixed) for each $i=1,2$ (see \cite{s2acc}).
\end{enumerate}
}
\end{remark}

\section{Block-Wise Strongly Unmixed, Accessible and Cohen-Macaulay properties}\label{secblocksu}
\medskip

The main aim of this section is to establish the connection between accessible, strongly unmixed properties of a graph and its corresponding blocks with whiskers. We start by proving some results regarding primary decomposition and unmixedness of $J_{G}$ as well as the accessibility of $G$.

\begin{proposition}\label{cutset}
Let $G=G_{1}\cup G_{2}$ be a graph such that $V(G_{1})\cap V(G_{2})=\{v\}$. Then $\mathfrak{C}(G)= \mathcal{A}\cup \mathcal{B}\cup\mathcal{C}\cup \mathcal{D}$, where
\begin{align*}
\mathcal{A}&= \{S_{1}\cup S_{2}\mid S_{1}\in \mathfrak{C}(G_{1}), S_{2}\in \mathfrak{C}(G_{2})\,\,\text{and}\,\, v\not\in S_{1}\cup S_{2}\}\\
\mathcal{B}&=\{T_{1}\cup S_{2}\mid v\in S_{2}\in \mathfrak{C}(G_{2})\,\,\text{and}\,\, T_{1}\in \mathfrak{C}(G_{1}\setminus \{v\})\}\\
\mathcal{C}&=\{S_{1}\cup T_{2}\mid v\in S_{1}\in \mathfrak{C}(G_{1})\,\,\text{and}\,\, T_{2}\in \mathfrak{C}(G_{2}\setminus \{v\})\}\\
\mathcal{D}&=\{T_{1}\cup T_{2}\cup \{v\} \mid \mathcal{N}_{G_{i}}(v)\not\subseteq T_{i}\in \mathfrak{C}(G_{i}\setminus\{v\}),\,\,\text{for}\,\, i\in\{1,2\}\}.
\end{align*}
\end{proposition}

\begin{proof}
Let $S\in \mathfrak{C}(G)$ be such that $v\not\in S$. Then $S=S_{1}\cup S_{2}$, where $S_{1}\subseteq V(G_{1})$ and $S_{2}\subseteq V(G_{2})$. For any $s\in S_{1}$, $s$ is a cut vertex in $G\setminus (S\setminus \{s\})$. Now $v\not\in S$ implies $v\in V(G\setminus(S\setminus \{s\}))$ and therefore, $s$ is also a cut vertex in $G_{1}\setminus (S_{1}\setminus \{s\})$. Hence $S_{1}\in \mathfrak{C}(G_{1})$ and by similar argument we get $S_{2}\in \mathfrak{C}(G_{2})$. Thus, $S\in\mathcal{A}$. Now assume $S\in \mathfrak{C}(G)$ and $v\in S$. Then we can write $S=T_{1}\cup T_{2}\cup \{v\}$, where $T_{i}\subseteq V(G_{i}\setminus\{v\})$ for $i=1,2$. Note that any $t\in T_{1}$ is a cut vertex of $G\setminus (S\setminus \{t\})$ which imply $t$ is a cut point of $G_{1}\setminus (T_{1}\cup \{v\}\setminus \{t\})=(G_{1}\setminus \{v\})\setminus (T_{1}\setminus \{t\})$. Therefore $T_{1}\in \mathfrak{C}(G_{1}\setminus \{v\})$ and similarly, we have $T_{2}\in \mathfrak{C}(G_{2}\setminus \{v\})$. Now we will consider few cases. Suppose $\mathcal{N}_{G_{i}}(v)\subseteq T_{i}$ for $i=1,2$. Then $v$ can not be a cut vertex in $G\setminus (S\setminus \{v\})$ which is a contradiction to the fact that $S\in \mathfrak{C}(G)$. Now assume $\mathcal{N}_{G_{1}}(v)\subseteq T_{1}$ but $\mathcal{N}_{G_{2}}(v)\not\subseteq T_{2}$. Set $S_{2}=T_{2}\cup \{v\}$. Since $v$ is a cut vertex in $G\setminus (T_{1}\cup T_{2})$ and $\mathcal{N}_{G_{1}}(v)\subseteq T_{1}$, it is easy to observe that $v$ is a cut vertex in $G_{2}\setminus (S_{2}\setminus \{v\})$. Also, $T_{2}\in \mathfrak{C}(G_{2}\setminus \{v\})$ implies any $t\in T_{2}$ is a cut point of $(G_{2}\setminus \{v\})\setminus ( T_{2}\setminus \{t\})=G_{2}\setminus (S_{2}\setminus \{t\})$. Therefore, we have $T_{1}\in \mathfrak{C}(G_{1}\setminus \{v\})$ and $S_{2}\in \mathfrak{C}(G_{2})$. Hence $S=T_{1}\cup S_{2} \in \mathcal{B}$. Similarly, if $\mathcal{N}_{G_{1}}(v)\not\subseteq T_{1}$ and $\mathcal{N}_{G_{2}}(v)\subseteq T_{2}$, then we get $S\in \mathcal{C}$. Consider $\mathcal{N}_{G_{i}}(v)\not\subseteq T_{i}$ for $i=1,2$. Then $S\in\mathcal{D}$ is clear from definition of $\mathcal{D}$\par

Conversely, let $S\in \mathcal{A}$. Then $S=S_{1}\cup S_{2}$ with $S_{1}\in \mathfrak{C}(G_{1})$, $S_{2}\in \mathfrak{C}(G_{2})$, and $v\not\in S$. Now, any $s_{i}\in S_{i}$ is a cut vertex in $G_{i}\setminus (S_{i}\setminus \{s_{i}\})$ which imply $s_{i}$ is a cut vertex in $G\setminus(S\setminus \{s_{i}\})$ as $v\not\in S$, where $i\in\{1,2\}$. Thus, $S\in \mathfrak{C}(G)$. Let $S\in \mathcal{B}$. Then $S=T_{1}\cup S_{2}$ such that $T_{1}\in \mathfrak{C}(G_{1}\setminus\{v\})$ and $v\in S_{2}\in \mathfrak{C}(G_{2})$. Let $t\in T_{1}$ be any vertex. Then $t$ is a cut vertex in $(G_{1}\setminus \{v\})\setminus (T_{1}\setminus\{t\})$ and  since $v\in S_{2}$, $t$ is also a cut point of $G\setminus (S\setminus \{t\})$. Again, any $s\in S_{2}$ is a cut vertex in $G_{2}\setminus (S_{2}\setminus \{s\})$. Since $T_{1}\cap V(G_{2})=\phi$, $s$ is a cut point of $G\setminus (S\setminus \{s\})$ too. Hence, $S\in \mathfrak{C}(G)$. Similarly, we have $\mathcal{C}\subseteq \mathfrak{C}(G)$. For $S\in \mathcal{D}$, we have $S=T_{1}\cup T_{2}\cup \{v\}$, where $\mathcal{N}_{G_{i}}(v)\not\subseteq T_{i}\in \mathfrak{C}(G_{i}\setminus\{v\})$, for $i\in\{1,2\}$. Then any $t\in T_{1}\cup T_{2}$ is clearly a cut vertex of $G\setminus (S\setminus \{t\})$. Now, since $\mathcal{N}_{G_{i}}(v)\not\subseteq T_{i}$, there exists $v_{i}\in \mathcal{N}_{G_{i}}(v)$ such that $v_{i}\in V(G_{i}\setminus T_{i})$, for $i=1,2$. Therefore, $v$ is a cut point in $G\setminus (S\setminus\{v\})$ and so, $S\in \mathfrak{C}(G)$. Hence, $\mathfrak{C}(G)=\mathcal{A}\cup\mathcal{B}\cup\mathcal{C}\cup\mathcal{D}$.
\end{proof}

\begin{corollary}\label{unmixed}
Let $G=G_{1}\cup G_{2}$ be such that $V(G_{1})\cap V(G_{2})=\{v\}$. If the following conditions hold:
\begin{enumerate}[(i)]
\item $J_{G\setminus\{v\}}$ is unmixed;
\item for $S_{i}\in \mathfrak{C}(G_{i})$ with $v\not\in S_{i}$, we have $c_{G_{i}}(S_{i})=\vert S_{i}\vert+1$ and for $S_{i}\in \mathfrak{C}(G_{i})$ with $v\in S_{i}$, we have $c_{G_{i}}(S_{i})=\vert S_{i}\vert$, where $i=1,2$;
\end{enumerate}
then $J_{G}$ is unmixed. 
\end{corollary}

\begin{proof}
We have $\mathfrak{C}(G)=\mathcal{A}\cup\mathcal{B}\cup\mathcal{C}\cup\mathcal{D}$ as described in Proposition \ref{cutset}. Let $T\in \mathcal{A}$. Then $T=S_{1}\cup S_{2}$, where $v\not\in S_{i}\in \mathfrak{C}(G_{i})$. Therefore by given condition (ii), we have $c_{G_{i}}(S_{i})=\vert S_{i}\vert +1$ and one component of $G_{i}\setminus S_{i}$ contains $v$ for $i=1,2$. Thus, $c_{G}(T)= \vert T\vert +1$ is clear. Let $T\in \mathcal{B}$. Then $T=T_{1}\cup S_{2}$, where $v\in S_{2}\in \mathfrak{C}(G_{2})$ and $T_{1}\in \mathfrak{C}(G_{1}\setminus \{v\})$. By condition (i) and (ii), we have $c_{G_{1}\setminus \{v\}}(T_{1})=\vert T_{1}\vert +1$ and $c_{G_{2}}(S_{2})=\vert S_{2}\vert$. So it is clear that $c_{G}(T)=\vert T\vert+1$. The case of $T\in \mathcal{C}$ is similar. If $T\in \mathcal{D}$, then $T= T_{1}\cup T_{2}\cup \{v\}$ with $\mathcal{N}_{G_{i}}(v)\not\subseteq T_{i}\in \mathfrak{C}(G_{i}\setminus\{v\})$. By condition (i), $c_{G}(T)=\vert T_{1}\vert +1+\vert T_{2}\vert +1=\vert T\vert+1$. Hence we can conclude that $J_{G}$ is unmixed.
\end{proof}

\begin{proposition}
Let $G$ be an accessible graph and $v_{1},v_{2}$ be two cut vertices of $G$ such that $\{v_{1},v_{2}\}\not\in E(G)$. If $G\setminus \{v_{1}\}$ and $G\setminus \{v_{2}\}$ are accessible, then $G\setminus \{v_{1},v_{2}\}$ is accessible. 
\end{proposition}

\begin{proof}
If $G$ is disconnected and $v_{1}$, $v_{2}$ belong to two different connected components of $G$, then by definition of accessibility $G\setminus \{v_{1},v_{2}\}$ is accessible. So we may assume $G$ is connected. Since $v_{1}$ is not adjacent to $v_{2}$, $\{v_{1},v_{2}\}\in \mathfrak{C}(G)$ and $J_{G}$ being unmixed, $c_{G}(\{v_{1},v_{2}\})=3$. Let $H_{1}$, $H^{\prime}$ be two connected components of $G\setminus \{v_{1}\}$ with $v_{2}\in V(H_{1})$ and $H_{2}$, $H^{\prime\prime}$ be two connected components of $G\setminus \{v_{2}\}$ with $v_{1}\in V(H_{2})$. Then $H_{1}\setminus (V(H^{\prime\prime})\cup \{v_{2}\})=H_{2}\setminus (V(H^{\prime})\cup \{v_{1}\})=H$ (say) and $H^{\prime}, H, H^{\prime\prime}$ are three connected components of $G\setminus \{v_{1},v_{2}\}$. Since $G\setminus \{v_{i}\}$ is accessible for $i=1,2$, we have $H_{1},H^{\prime},H_{2},H^{\prime\prime}$ are accessible. So, it is enough to show that $H$ is accessible. Let $T\in \mathfrak{C}(H)$. Suppose $\mathcal{N}_{H_{1}}(v_{1})\subseteq T$. Since $G$ and $G\setminus \{v_{1}\}$ are accessible, by Proposition \ref{accpropunm}, $T\not\in \mathfrak{C}(G\setminus \{v_{1}\})$ i.e., $T\not \in \mathfrak{C}(H_{1})$. But $T\in \mathfrak{C}(H)$ which imply there must exists $w\in \mathcal{N}_{H}(v_{2})$ such that $w\not\in T$. Then $v_{2}$ is a cut point in $H_{1}\setminus(T\cup\{v_{2}\}\setminus \{v_{2}\})$ and obviously, every $t\in T$ is a cut vertex in $H_{1}\setminus (T\cup\{v_{2}\}\setminus \{t\})$. Therefore, $T\cup\{v_{2}\}\in \mathfrak{C}(H_{1})$ and $H_{1}$ being unmixed, we have $c_{H_{1}}(T\cup\{v_{2}\})=\vert T\vert +2$. Thus, $c_{H}(T)=\vert T\vert +1$. Similarly, if $\mathcal{N}_{H_{2}}(v_{2})\subseteq T$, then also we get $c_{H}(T)=\vert T\vert +1$. Now, assume $\mathcal{N}_{H_{1}}(v_{1})\not\subseteq T$ and $\mathcal{N}_{H_{2}}(v_{2})\not\subseteq T$. Since $\{v_{1},v_{2}\}\not\in E(G)$, $v_{1}$ is a cut vertex in $G\setminus (T\cup \{v_{2}\})$ and $v_{2}$ is a cut vertex in $G\setminus (T\cup \{v_{1}\})$. Thus, $T\cup\{v_{1},v_{2}\}\in \mathfrak{C}(G)$ and $c_{G}(T\cup\{v_{1},v_{2}\})=\vert T\vert +3$ as $J_{G}$ is unmixed. Since $c_{G}(\{v_{1},v_{2}\})=3$ and $T\subseteq V(H)$, we have $c_{H}(T)=\vert T\vert +1$. Hence $J_{H}$ is unmixed and so, $J_{G\setminus \{v_{1},v_{2}\}}$ is unmixed. Since $H_{1}$ is connected accessible graph and $v_{2}$ is a cut vertex in $H_{1}$, by \cite[Proposition 5.14]{acc}, $H_{1}\setminus \{v_{2}\}$ is accessible. Therefore we get $H$ is accessible and so is $G\setminus\{v_{1},v_{2}\}$.
\end{proof}

\begin{proposition}\label{freacc}
Let $G$ be a simple graph such that $J_{G}$ is unmixed and $G\setminus\{v\}$ is accessible for a free vertex $v\in V(G)$. Then $G$ is accessible. 
\end{proposition}

\begin{proof}
From the proof of \cite[Lemma 2.2]{raufrin}, we have $\mathfrak{C}(G\setminus\{v\})\subseteq \mathfrak{C}(G)$ for a free vertex $v$ of $G$. Let $T\in \mathfrak{C}(G)$. Then $v\not\in T$ as $v$ is a free vertex of $G$. Set $H=G\setminus\{v\}$. We will consider two cases:\par 

\noindent\textbf{Case-I:} Let $\mathcal{N}_{G}(v)\not\subseteq T$. Then $T\in \mathfrak{C}(H)$ by \cite[Lemma 2.2]{raufrin}. Since $H$ is accessible, $T$ is accessible as a cutset of $H$ as well as of $G$.\par 

\noindent\textbf{Case-II:} Assume $\mathcal{N}_{G}(v)\subseteq T$. Then $S=\mathcal{N}_{G}(v)\in \mathfrak{C}(G)$ is clear. As $J_{G}$ is unmixed $c_{G}(S)=\vert S\vert +1$. Then $H\setminus S$ has $\vert S\vert$ connected components. So, $S\not\in \mathfrak{C}(H)$ as $J_{H}$ is unmixed. Let $S=\{s_{1},\ldots,s_{k}\}$ and connected components of $G\setminus S$ are $A_{s_{1}},\ldots, A_{s_{k}},\{v\}$. Assume $k>1$. Since $S\not\in \mathfrak{C}(H)$, there exists an $s_{i}$ which is not a cut vertex in $H\setminus (S\setminus\{s_{i}\})$, but $s_{i}$ is a cut vertex in $G\setminus (S\setminus\{s_{i}\})$. Therefore, $s_{i}$ is adjacent to only one connected components of $H\setminus S$, say $A_{s_{i}}$. For simplicity of notations, let $i=1$ i.e., $s_{1}$ is a cut vertex in $G$ and it is adjacent to only $A_{s_{1}}$ in $H\setminus (S\setminus\{s_{1}\})$. Then clearly $S\setminus\{s_{1}\}\in \mathfrak{C}(H)$ and thus $S$ is accessible. Let $V(A_{s_{1}})\cap T=T^{\prime}$. Since $s_{1}$ and $A_{s_{1}}$ are not adjacent to $A_{s_{2}},\ldots, A_{s_{k}}$, we have $T^{\prime\prime}=T\setminus (T^{\prime}\cup \{s_{1}\})\in \mathfrak{C}(H)$ and $T^{\prime}\cup\{s_{1}\}\in \mathfrak{C}(G)$. We have $ V(A_{s_{1}})\setminus T^{\prime}\neq \phi$ and $\mathcal{N}_{A_{s_{1}}}(s_{1})\not\subseteq T^{\prime}$ otherwise, $s_{1}$ can not be a cut vertex of $G\setminus T^{\prime}$. Since $k>1$, $s_{1}$ is a cut point in $H\setminus T^{\prime}$. Let $T^{\prime}\neq \phi$. Then every $t\in T^{\prime}$ is a cut point of $H\setminus (T^{\prime}\cup\{s_{1}\}\setminus \{t\})$. Hence $T^{\prime}\cup\{s_{1}\}\in \mathfrak{C}(H)$. By accessibility of $H$ and Proposition \ref{accnoncut}, 
\begin{equation}\label{eq1}
(T^{\prime}\cup\{s_{1}\})\setminus \{t^{\prime}\}\in \mathfrak{C}(H)
\end{equation} 
for some $t^{\prime}\in T^{\prime}$. Now $T^{\prime\prime}\in \mathfrak{C}(H)$ and $s_{1}\not\in T^{\prime\prime}$ is a cut vertex of $G$ with $\mathcal{N}_{H}(s_{1})\not\subseteq T^{\prime\prime}$ together imply $T^{\prime\prime}\cup
\{s_{1}\}\in \mathfrak{C}(G)$. So, we have by \ref{eq1},
$$(T^{\prime\prime}\cup\{s_{1}\}\cup T^{\prime})\setminus\{t^{\prime}\}=T\setminus\{t^{\prime}\}\in \mathfrak{C}(G).$$ 
Thus, $T$ is accessible. If $T^{\prime}=\phi$, then $T^{\prime\prime}\in \mathfrak{C}(H)$ implies $T\setminus\{s_{1}\}=T^{\prime\prime}\in \mathfrak{C}(G)$ and so, $T$ is accessible. Now assume $k=1$. Then any $T\in \mathfrak{C}(G)$ with $s_{1}\not\in T$ implies $T\in \mathfrak{C}(H)$ and we are done. Let $T\in \mathfrak{C}(G)$ such that $s_{1}\in T$. Then $\mathcal{N}_{G}(s_{1})\setminus\{v\}\not\subseteq T$. Suppose $T\setminus\{s_{1}\}\not\in \mathfrak{C}(H)$. Then there exists $t\in T\setminus\{s_{1}\}$ such that $t$ is not a cut point of $(G\setminus\{v\})\setminus (T\cup\{s_{1}\}\setminus\{t\})=H_{t}$. Now $s_{1},t\in V(H_{t})$ is clear and $t$ is a cut point of $G\setminus(T\setminus\{t\})$. So, there exists  $x,y\in \mathcal{N}_{H}(s_{1})$ such that $x,y$ are not connected in $H\setminus T$. Therefore, $T\in \mathfrak{C}(H)$. Thus, for $k=1$ either $T\in \mathfrak{C}(H)$ or $T\setminus\{s_{1}\}\in \mathfrak{C}(H)$ and for both the cases $T$ is accessible as a cutset of $G$. Hence $G$ is accessible.
\end{proof}

\begin{lemma}\label{Gvunmacc}
Let $G=G_{1}\cup G_{2}$ be a graph such that $V(G_{1})\cap V(G_{2})=\{v\}$. Then $\mathfrak{C}(G_{v})=\mathfrak{C}((G_{1})_{v})\cup \mathfrak{C}((G_{2})_{v})$. In particular, $J_{(G_{1})_{v}}$ and $J_{(G_{2})_{v}}$ are unmixed (resp, accessible) if and only if $J_{G_{v}}$ is unmixed (resp, accessible).
\end{lemma}

\begin{proof}
Let $T\in \mathfrak{C}(G_{v})$. Set $T\cap V(G_{i})=T_{i}$ for $i=1,2$. Then $T=T_{1}\cup T_{2}$ and $v$ being a free vertex $v\not\in T$. Now, any $t\in T_{i}$ is a cut point of $G_{v}\setminus(T\setminus\{t\})$ imply $t$ is also a cut point in $(G_{i})_{v}\setminus (T_{i}\setminus\{t\})$, where $i=1,2$. Thus, $T_{i}\in \mathfrak{C}((G_{i})_{v})$ for $i=1,2$ and $\mathfrak{C}(G_{v})\subseteq \mathfrak{C}((G_{1})_{v})\cup \mathfrak{C}((G_{2})_{v})$.\par
Let $T_{i}\in \mathfrak{C}((G_{i})_{v})$ for $i=1,2$. Since $v$ is a free vertex in $(G_{i})_{v}$, $v\not\in T_{i}$ and each $t\in T_{1}\cup T_{2}$ is a cut vertex in $G_{v}\setminus(T_{1}\cup T_{2} \setminus\{t\})$, where $i\in\{1,2\}$. Therefore, $T_{1}\cup T_{2}\in \mathfrak{C}(G_{v})$ and $\mathfrak{C}(G_{v})=\mathfrak{C}((G_{1})_{v})\cup \mathfrak{C}((G_{2})_{v})$. \par 
Note that, $(G_{i})_{v}\setminus T_{i}$ contains a connected component containing $v$, where $T_{i}\in \mathfrak{C}((G_{i})_{v})$ and $i=1,2$. Thus, we have
$$c_{G_{v}}(T_{1}\cup T_{2})=c_{(G_{1})_{v}}(T_{1})+ c_{(G_{2})_{v}}(T_{2})-1.$$
Hence, $J_{(G_{1})_{v}}$ and $J_{(G_{2})_{v}}$ are unmixed (resp, accessible) if and only if $J_{G_{v}}$ is unmixed (resp, accessible).
\end{proof}

\begin{proposition}\label{Gvvsu}
Let $G=G_{1}\cup G_{2}$ be a graph such that $V(G_{1})\cap V(G_{2})=\{v\}$. If $J_{G_{v}}$ is strongly unmixed, then $J_{(G_{1})_{v}}$ and $J_{(G_{2})_{v}}$ are strongly unmixed.
\end{proposition}

\begin{proof}
Let $\vert V(G_{v})\vert=n$ and we will proceed by induction on $n$. If one of $G_{1}$ or $G_{2}$ is a graph with only one vertex which is $v$, then the result holds trivially. So assume $G_{1}, G_{2}$ are not empty (i.e a graph without edges). Then the base case will be $n=3$ and $G$ is a $P_{2}$ with $(G_{i})_{v}=K_{2}$ for each $i\in\{1,2\}$. Therefore the result follows by definition of strongly unmixed. Assume $n>3$. If $G_{v}$ is complete, then $(G_{1})_{v}, (G_{2})_{v}$ are both complete and the result follows by definition. If $G_{v}$ is non-complete, then there exists a cut vertex $u$ of $G_{v}$ such that $G_{v}\setminus\{u\}, (G_{v})_{u}, (G_{v})_{u}\setminus \{u\}$ are strongly unmixed. Note that $v$ being a free vertex of $G_{v}$, it can not be a cut vertex of $G_{v}$. Let us suppose $u\in V(G_{1})$. By Lemma \ref{Gvunmacc}, $J_{(G_{i})_{v}}$ is unmixed as $J_{G_{v}}$ is unmixed for $i=1,2$. Let $G_{v}\setminus\{u\}=H\sqcup H_{2}$, where $V(G_{2})\subseteq V(H)$. Then $(G_{1})_{v}\setminus\{u\}=H_{1}\sqcup H_{2}$, where $H_{1}=H\setminus V(G_{2}\setminus\{v\})$. Now $H, H_{2}$ are strongly unmixed and $H=(H_{1}\cup G_{2})_{v}$. Since $H$ has less than $n$ vertices, by induction hypothesis the binomial edge ideals of $(G_{2})_{v}$ and $(H_{1})_{v}=H_{1}$ are strongly unmixed. Thus, $(G_{1})_{v}\setminus \{u\}$ is strongly unmixed. Also, we observe that
$$(G_{v})_{u}\setminus\{u\}= \big(((G_{1})_{u}\setminus\{u\})\cup G_{2}\big)_{v}.$$
Since $V((G_{1})_{u}\setminus\{u\})\cap V(G_{2})=\{v\}$, using induction hypothesis, we get the binomial edge ideals of $((G_{1})_{u}\setminus\{u\})_{v}$ and $(G_{2})_{v}$ are strongly unmixed. Note that $((G_{1})_{u}\setminus\{u\})_{v}=((G_{1})_{v})_{u}\setminus\{u\}$. Therefore, $J_{((G_{1})_{v})_{u}\setminus\{u\}}$ is strongly unmixed and by Lemma \ref{freeGsu}, $J_{((G_{1})_{v})_{u}}$ is also strongly unmixed. Hence $J_{(G_{1})_{v}}$ and $J_{(G_{2})_{v}}$ are strongly unmixed. The case of $u\in V(G_{2})$ is similar.
\end{proof}

\begin{proposition}\label{Gvsu}
Let $G=G_{1}\cup G_{2}$ be a graph with $V(G_{1})\cap V(G_{2})=\{v\}$. If $J_{(G_{1})_{v}}$ and $J_{(G_{2})_{v}}$ are strongly unmixed, then $J_{G_{v}}$ is strongly unmixed.
\end{proposition}

\begin{proof}
We proceed by induction on $n$, where $n=\vert V(G_{v})\vert$. If one of $G_{1}$ and $G_{2}$ is empty, then the result follows trivially. So $n=3$ is the base case for which $G_{v}$ is complete and hence $J_{G_{v}}$ is strongly unmixed. If $(G_{1})_{v}$ and $(G_{2})_{v}$ are complete, then $G_{v}$ is also complete and so $J_{G_{v}}$ is strongly unmixed. Assume at least one of $(G_{1})_{v}$ and $(G_{2})_{v}$ is non-complete, say $(G_{1})_{v}$. By Lemma \ref{Gvunmacc}, $J_{G_{v}}$ is unmixed. For simplicity of notation we set $H^{i}=(G_{i})_{v}$, where $i=1,2$. Since $H^{1}$ is non-complete and $J_{H^{1}}$ is strongly unmixed, there exists a cut vertex $u$ of $H^{1}$ for which $J_{H^{1}\setminus\{u\}}, J_{H^{1}_{u}}$, and $J_{H^{1}_{u}\setminus\{u\}}$ are strongly unmixed. Let $H^{1}\setminus\{u\}= H^{\prime}\sqcup H^{\prime\prime}$, where $v\in V(H^{\prime})$. Then $G_{v}\setminus\{u\}=H\sqcup H^{\prime\prime}$, where $V(G_{2})\subseteq V(H)$. Note that,
\begin{align*}
H&= (H^{\prime}\cup G_{2})_{v}\,\,\, \text{if}\,\, \mathcal{N}_{G_{1}}(v)\neq\{u\},\\
&= (G_{2})_{v}\,\,\,\,\,\,\,\,\,\,\,\,\,\,\,\, \text{if}\,\, \mathcal{N}_{G_{1}}(v)=\{u\}.
\end{align*}
Now $J_{H^{\prime}}=J_{H^{\prime}_{v}}$ and $J_{(G_{2})_{v}}$ are strongly unmixed. Thus, using induction hypothesis, we get $J_{H}$ is strongly unmixed. $J_{H^{\prime\prime}}$ being also strongly unmixed, $J_{G_{v}\setminus\{u\}}$ is strongly unmixed. Consider $(G_{v})_{u}\setminus\{u\}$ and note that
$$ (G_{v})_{u}\setminus\{u\}= (H^{1}_{u}\setminus\{u\}\cup G_{2})_{v}.$$
Now $(H^{1}_{u}\setminus\{u\})_{v}=H^{1}_{u}\setminus\{u\}$ and hence by induction hypothesis, $J_{(G_{v})_{u}\setminus\{u\}}$ is strongly unmixed. Since $u$ is a free vertex in $(G_{v})_{u}$, by Lemma \ref{freeGsu}, we have $J_{(G_{v})_{u}}$ is strongly unmixed.
\end{proof}

\begin{definition}{\rm
Let $G$ be a connected graph such that $J_{G}$ is unmixed and $B$ be a block of $G$. Let $V=\{v_{1},\ldots,v_{k}\}$ be the set of cut vertices of $G$ belonging to $V(B)$. Then we can write
\begin{align}\label{blockdec}
G=B\cup\big(\bigcup_{i=1}^k G_{i}\big),
\end{align}
where $V(G_{i})\cap V(B)=\{v_{i}\}$ for each $1\leq i\leq k$, and the connected components of $G\setminus V$ are $B\setminus V$ (may be empty), $G_{1}\setminus\{v_{1}\},\ldots, G_{k}\setminus\{v_{k}\}$. 
}
\end{definition}

Considering the decomposition \ref{blockdec}, we define a new graph $\overline{B}^{W}$, where $W=\{v_{s_{1}},\ldots,v_{s_{r}}\}\subseteq V$ such that
\begin{enumerate}
\item[$\bullet$] $V(\overline{B}^{W})=V(B)\cup \{f_{v_{{s_{1}}}},\ldots, f_{v_{{s_{r}}}}\},$

\item[$\bullet$] $E(\overline{B}^W)=E(B)\cup \big(\bigcup_{v_{i}\not\in W} E(G_{i})\big)\cup \{\{v_{s_{i}},f_{v_{{s_{i}}}}\}\mid i=1,\ldots, r\}.$
\end{enumerate}
By $\overline{B}$ \textit{with respect to} $G$, we mean $\overline{B}^V$ and call it the \textit{block with whiskers} of $G$ (Sometimes we write only $\overline{B}$ if the graph is clear from the context or sometimes by $\overline{B}$ we mean a block attaching with some whiskers). In simple words, $\overline{B}$ is the graph attaching whiskers to all the cut vertices $v_{i}$ of $G$ belong to $V(B)$ replacing $G_{i}$'s.
\medskip

\begin{theorem}\label{baracc}
Let $G$ be an accessible graph and $B$ be any block of $G$. Let $V=\{v_{1},\ldots,v_{k}\}$ be the set of cut vertices of $G$ belong to $V(B)$. Then for any $W\subseteq V$, $\overline{B}^W$ is accessible. In particular, $\overline{B}$  is accessible.
\end{theorem}

\begin{proof}
Let $T\in \mathfrak{C}(\overline{B}^W)$ and without loss of generality $W=\{v_{1},\ldots,v_{r}\}$, where $r\leq k$. Assume the decomposition of $G$ with respect to $B$ is
$$G=B\cup\big(\bigcup_{i=1}^k G_{i}\big),$$
where $V(G_{i})\cap V(B)=\{v_{i}\}$ for each $i\in [k]$, and the connected components of $G\setminus V$ are $B\setminus V$ (may be empty), $G_{1}\setminus\{v_{1}\},\ldots, G_{k}\setminus\{v_{k}\}$. Then the connected components of $\overline{B}^W\setminus V$ are $B\setminus V$ (may be empty), $\{f_{v_{1}}\},\ldots,\{f_{v_{r}}\}, G_{r+1}\setminus\{v_{r+1}\},\ldots, G_{k}\setminus\{v_{k}\}$. Let $A_{1},\ldots, A_{s}$ be the connected components of $\overline{B}^W\setminus T$. Now for each $A_{i}$ we consider an induced subgraph $A^{\prime}_{i}$ of $G$ as follows
\begin{align*}
 \bullet\,\,\,\,\, V(A^{\prime}_{i})&=(V(A_{i})\setminus F) \cup \bigg(\bigcup_{f_{v_{j}}\in V(A)} (V(G_{j})\setminus\{v_{j}\})\bigg);\\
 \bullet\,\,\,\,\, E(A^{\prime}_{i})&=E(A_{i}\setminus F)\cup\bigg(\bigcup_{\{v_{j},f_{v_{j}}\}\in E(A)} E(G_{j})\bigg)\,\,\,\,\text{if}\,\,\, V(A)\not\subseteq F,\\
 &=E(G_{j}\setminus\{v_{j}\})\,\,\,\, \text{if}\,\,\, V(A_{i})=\{f_{v_{j}}\};
 \end{align*}
where $F=\{f_{v_{1}},\ldots,f_{v_{r}}\}$. Then it is easy to see that $A^{\prime}_{1},\ldots,A^{\prime}_{s}$ are the only connected components of $G\setminus T$. Now $T\in \mathfrak{C}(\overline{B}^W)$ implies $T\in \mathfrak{C}(G)$ and also, we have shown that $c_{\overline{B}^W}(T)=c_{G}(T)$. Therefore, $J_{\overline{B}^W}$ is unmixed as $J_{G}$ is unmixed. Note that if $T\in \mathfrak{C}(G)$ and $T\subseteq V(\overline{B}^W)\setminus F$, then $T\in \mathfrak{C}(\overline{B}^W)$ also. Since each $f_{v_{i}}$ for $i=1,\ldots,r$ is a free vertex of $\overline{B}^W$, $T\in \mathfrak{C}(\overline{B}^W)$ gives $T\subseteq V(\overline{B}^W)\setminus F$. Hence $\overline{B}^W$ is accessible as $G$ is so.
\end{proof}

\begin{theorem}\label{blocksu}
Let $G$ be a graph such that $J_{G}$ is strongly unmixed. Let $B$ be a block of $G$ and $V$ be the set of cut vertices of $G$ belonging to $V(B)$. Then for any $W\subseteq V$, $J_{\overline{B}^W}$ is strongly unmixed. In particular, $J_{\overline{B}}$ is strongly unmixed.
\end{theorem}

\begin{proof}
We will go ahead by induction on the number of vertices $n$ of $G$. By Theorem \ref{baracc}, $J_{\overline{B}^{W}}$ is unmixed. If $G$ is complete then there is nothing to proof, otherwise there exists a cut vertex $v$ of $G$ such that $J_{G\setminus\{v\}}, J_{G_{v}}$ and $J_{G_{v}\setminus\{v\}}$ are strongly unmixed. Let $V=\{v_{1},\ldots, v_{k}\}$ and \ref{blockdec} is the decomposition of $G$ with respect to $B$.\par 

\noindent \textbf{Case-I:} Suppose $v\not\in V(B)$. Without loss of generality assume $W=\{v_{1},\ldots, v_{r}\}$, where $r\leq k$. If $v\in V(G_{i})$ for $i\in\{1,\ldots,r\}$, then by induction hypothesis $J_{\overline{B}^W}$ is strongly unmixed. Now, assume $v\not\in V(G_{i})$ for $i=1,\ldots,r$. Then $v$ is a cut vertex in $G_{\overline{B}^W}$. Therefore, $J_{G\setminus\{v\}}$ is strongly unmixed implies that $J_{\overline{B}^{W}\setminus\{v\}}$ is strongly unmixed by induction. Again by induction hypothesis, $J_{G_{v}\setminus\{v\}}$ is strongly unmixed imply $J_{(\overline{B}^W)_{v}\setminus\{v\}}$ is strongly unmixed and by Lemma \ref{freeGsu}, $J_{(\overline{B}^W)_{v}}$ is also strongly unmixed. Hence $J_{\overline{B}^W}$ is strongly unmixed.
\par 

\noindent\textbf{Case-II:} Let $v=v_{i}\in V(B)$. Let $\overline{B}^{W}\setminus\{v\}= H^{\prime}\sqcup H^{\prime\prime}$, where $V(B)\setminus\{v\}\subseteq H^{\prime\prime}$. Then $H^{\prime}= G_{i}\setminus\{v\}$ or $H^{\prime}=\{f_{v}\}$. Since, $J_{G\setminus\{v\}}$ is strongly unmixed, $J_{H^{\prime}}$ is strongly unmixed and by induction hypothesis $J_{H^{\prime\prime}}$ is strongly unmixed. Therefore $J_{\overline{B}^{W}\setminus\{v\}}$ is strongly unmixed. If $H^{\prime}= G_{i}\setminus\{v\}$, then by induction we have $J_{(\overline{B}^{W})_{v}\setminus\{v\}}$ is strongly unmixed and thus, by Lemma \ref{freeGsu}, $J_{(\overline{B}^{W})_{v}}$ is strongly unmixed. Set $\overline{B}^{W}\setminus (V(G_{i})\setminus\{v\})=H$. Then $H\setminus\{v\}=H^{\prime\prime}$. In this case observe that,
$$(G)_{v}=(H\cup G_{i})_{v}.$$
Then by Proposition \ref{Gvvsu}, $J_{H_{v}}$ is strongly unmixed. Now suppose $H^{\prime}=\{f_{v}\}$. Then observe that $\overline{B}^{W}\setminus\{v\}=H^{\prime\prime}\sqcup \{f_{v}\}$ and $
(\overline{B}^{W})_{v}\setminus\{v\}\simeq H_{v}$ and so $J_{(\overline{B}^{W})_{v}\setminus\{v\}}$ is strongly unmixed. By Lemma \ref{freeGsu}, $J_{(\overline{B}^{W})_{v}}$ is strongly unmixed and hence from definition, $J_{\overline{B}^W}$ is strongly unmixed.
\end{proof}

\begin{lemma}\label{freunm}
Let $G$ be a graph and $v\in V(G)$ is a free vertex of $G$. Then the following hold.
\begin{enumerate}[(i)]
\item If $\mathcal{N}_{G}(v)\not\in \mathfrak{C}(G)$, then $J_{G}$ is unmixed if and only if $J_{G\setminus \{v\}}$ is unmixed.
\item If for all $T\in \mathfrak{C}(G)$ with $\mathcal{N}_{G}(v)\subseteq T$ we have $c_{G}(T)=\vert T\vert+1$ and $J_{G\setminus\{v\}}$ is unmixed, then $J_{G}$ is unmixed.
\end{enumerate}
\end{lemma}

\begin{proof}
(i): Suppose $\mathcal{N}_{G}(v)\subseteq T$ for some $T\in \mathfrak{C}(G)$. Then each $s\in \mathcal{N}_{G}(v)$ is a cut point of $G\setminus (T\setminus\{s\})$. So, there exists $w\in \mathcal{N}_{G}(s)\setminus \{v\}$ such that $w\not\in T$. Therefore $s$ is a cut point of $G\setminus (\mathcal{N}_{G}(v)\setminus\{s\})$ for every $s\in \mathcal{N}_{G}(v)$. Hence, $\mathcal{N}_{G}(v)\in \mathfrak{C}(G)$. So, $\mathcal{N}_{G}(v)\not\in \mathfrak{C}(G)$ implies $\mathcal{N}_{G}(v)\not\subseteq T$ for all $T\in \mathfrak{C}(G)$. Then by \cite[Lemma 2.2]{raufrin}, we have $T\in \mathfrak{C}(G)$ if and only if $T\in \mathfrak{C}(G\setminus \{v\})$. Also it is easy to verify that $c_{G}(T)=c_{G\setminus\{v\}}(T)$ for all $T\in \mathfrak{C}(G)=\mathfrak{C}(G\setminus\{v\})$. Hence $J_{G}$ is unmixed if and only if $J_{G\setminus \{v\}}$ is unmixed.\par 
(ii): Let $T\in \mathfrak{C}(G)$ with $\mathcal{N}_{G}(v)\not\subseteq T$. Then by \cite[Lemma 2.2]{raufrin}, $T\in \mathfrak{C}(G\setminus\{v\})$ and notice that a connected component of $(G\setminus\{v\})\setminus T$ contains a vertex $w\in \mathcal{N}_{G}(v)$. Thus, we have $c_{G}(T)=c_{G\setminus\{v\}}(T)$. Since $J_{G\setminus\{v\}}$ is unmixed, by the given hypothesis, for all $T\in \mathfrak{C}(G)$ we get $c_{G}(T)=\vert T\vert +1$ and hence, $J_{G}$ is unmixed.
\end{proof}

\begin{proposition}\label{freesu}
Let $G$ be a graph and $v\in V(G)$ be a free vertex of $G$. If $J_{G}$ is strongly unmixed and there exists no cutset $T\in \mathfrak{C}(G)$ such that $\mathcal{N}_{G}(v)\subseteq T$, then $J_{G\setminus\{v\}}$ is strongly unmixed.
\end{proposition}

\begin{proof}
Let $G\setminus\{v\}=H$. Since $\mathcal{N}_{G}(v)\not\in \mathfrak{C}(G)$ and $v$ is a free vertex of $G$, by Lemma \ref{freunm}, $J_{H}$ is unmixed. We proceed by induction on the number of vertices $n$ of $G$. For $n=1,2$, the result holds trivially. If $G$ is complete, then $H$ is also complete and we are done. Suppose $G$ is not complete. Then there exists a cut vertex $u\in V(G)$ of $G$ for which $J_{G\setminus\{u\}}, J_{G_{u}}$, and $J_{G_{u}\setminus\{u\}}$ are strongly unmixed. Since $v$ is a free vertex in $G$, $v\in V(G\setminus\{u\})$ is a free vertex of $G\setminus\{u\}$. Note that $(G\setminus\{u\})\setminus\{v\}=H\setminus\{u\}$. If there exists $T\in \mathfrak{C}(G\setminus\{u\})$ such that $\mathcal{N}_{G\setminus\{u\}}(v)\subseteq T$, then by Proposition \ref{accpropunm}, $T\cup \{u\}\in \mathfrak{C}(G)$ which leads to a contradiction as $\mathcal{N}_{G}(v)\subseteq T\cup\{u\}$. Thus, $G\setminus\{u\}$ satisfies the given conditions and has less than $n$ vertices. Therefore, by induction hypothesis, $J_{H\setminus\{u\}}$ is strongly unmixed. Now, from (\cite{acc}, Lemma 4.5 and  Lemma 5.5), we have
\begin{align*}
\mathfrak{C}(G_{u}\setminus\{u\})&= \mathfrak{C}(G_{u})\setminus\{T\in \mathfrak{C}(G_{u})\mid \mathcal{N}_{G}(u)\subseteq T\}\\
\mathfrak{C}(G_{u})&=\{T\in \mathfrak{C}(G)\mid u\not\in T\}.
\end{align*}
Suppose there exists $T\in \mathfrak{C}(G_{u}\setminus\{u\})$ such that $\mathcal{N}_{G_{u}\setminus\{u\}}(v)\subseteq T$. Then $T\in \mathfrak{C}(G\setminus \{u\})$ is clear and by Proposition \ref{accpropunm}, $T\cup\{u\}\in \mathfrak{C}(G)$. Note that $\mathcal{N}_{G}(v)\setminus\{u\}\subseteq \mathcal{N}_{G_{u}\setminus\{u\}}(v)\subseteq T$ which imply $\mathcal{N}_{G}(v)\subseteq T\cup\{u\}\in \mathfrak{C}(G)$, a contradiction. Therefore, for all $T\in \mathfrak{C}(G_{u}\setminus\{u\})$ we have $\mathcal{N}_{G_{u}\setminus\{u\}}(v)\not\subseteq T$. Since $(G_{u}\setminus\{u\})\setminus\{v\}=H_{u}\setminus\{u\}$, by induction hypothesis $J_{H_{u}\setminus\{u\}}$ is strongly unmixed and by Lemma \ref{freeGsu}, $J_{H_{u}}$ is strongly unmixed.
\end{proof}

\begin{lemma}\label{barunm}
Let $G=G_{1}\cup G_{2}$ be such that $V(G_{1})\cap V(G_{2})=\{v\}$. Consider the graph $\overline{G_{i}}$ by attaching a whisker $\{v,f_{v}\}$ to $G_{i}$ at $v$. If $J_{G}$ is unmixed, then $J_{\overline{G_{i}}}$ is unmixed for $i=1,2$.
\end{lemma}

\begin{proof}
Without loss of generality, we assume $G$ is connected. Then $\overline{G_{i}}$ is connected for $i=1,2$. Let $T\in \mathfrak{C}(\overline{G_{1}})$. Then it is clear that $T\in \mathfrak{C}(G)$. Let $A_{1},\ldots,A_{k}$ be the connected components of $\overline{G_{1}}\setminus T$. One of components, say $A_{1}$, will contain the vertex $f_{v}$. We consider the graph $A^{\prime}_{1}$ as follows
\begin{align*}
 \bullet\,\,\,\,\, V(A^{\prime}_{1})&=(V(A_{1})\setminus \{f_{v}\}) \cup V(G_{2}\setminus\{v\});\\
 \bullet\,\,\,\,\, E(A^{\prime}_{1})&=E(A_{1}\setminus\{f_{v}\})\cup E(G_{2})\,\,\,\,\text{if}\,\,\, \{v,f_{v}\}\in E(A_{1}),\\
 &=E(G_{2}\setminus\{v\})\,\,\,\, \text{if}\,\,\, \{v,f_{v}\}\not\in E(A_{1})\,\, \text{i.e.},\,\, V(A_{1})=\{f_{v}\}.
\end{align*}
Then $A^{\prime}_{1}$ is connected and $A^{\prime}_{1}, A_{2},\ldots, A_{k}$ are the only connected components of $G\setminus T$. Therefore, $c_{\overline{G_{1}}}(T)=c_{G}(T)$ for all $T\in \mathfrak{C}(\overline{G_{1}})$ and so, $J_{\overline{G_{1}}}$ is unmixed as $J_{G}$ is so. Similarly, $J_{\overline{G_{2}}}$ is unmixed.
\end{proof}

\begin{lemma}\label{Gbaracc}
Let $G=G_{1}\cup G_{2}$ be such that $V(G_{1})\cap V(G_{2})=\{v\}$. Consider the graph $\overline{G_{i}}$ by attaching a whisker $\{v,f_{v}\}$ to $G_{i}$ at $v$. If $J_{G}$ is unmixed and $\overline{G_{1}}, \overline{G_{2}}$ are accessible, then $G$ is accessible.
\end{lemma}

\begin{proof}
Without loss of generality, we can assume $G$ is connected. Let $\mathfrak{C}(G)=\mathcal{A}\cup \mathcal{B}\cup \mathcal{C}\cup \mathcal{D}$ as described in Proposition \ref{cutset}. Let $T\in \mathcal{A}$. Then $T=S_{1}\cup S_{2}$, where $v\not\in S_{i}\in \mathfrak{C}(G_{i})$ for $i=1,2$. Clearly, $S_{i}\in \mathfrak{C}(\overline{G_{i}})$ and so, there exists $s_{i}\in S_{i}$ such that $S_{i}\setminus\{s_{i}\}\in \mathfrak{C}(\overline{G_{i}})$, where $i\in\{1,2\}$. Now, $v\not\in S_{1}\setminus\{s_{1}\}$ imply $S_{1}\setminus\{s_{1}\}\in \mathfrak{C}(G_{1})$ by \cite[Lemma 2.2]{raufrin}. Therefore, $T\setminus\{s_{1}\}\in \mathcal{A}$ and so, $T$ is accessible. Let $T=T_{1}\cup S_{2}\in\mathcal{B}$, where $T_{1}\in \mathfrak{C}(G_{1}\setminus\{v\})$ and $v\in S_{2}\in \mathfrak{C}(G_{2})$. If $\mathcal{N}_{G_{1}}(v)\subseteq T_{1}$, then $T_{1}\in \mathfrak{C}(\overline{G_{1}})$ and one connected component of $G_{1}\setminus T_{1}$ consists of only the edge $\{v,f_{v}\}$. Again, $S_{2}\in \mathfrak{C}(\overline{G_{2}})$. Since $J_{\overline{G_{i}}}$ is unmixed for $i=,1,2$, we have
$$c_{G}(T)=\vert T_{1}\vert+1+\vert S_{2}\vert +1 -2=\vert T\vert,$$
which gives a contradiction to the fact that $J_{G}$ is unmixed. Thus, $\mathcal{N}_{G_{1}}(v)\not\subseteq T_{1}$ and so, $T\in \mathcal{D}$ also. Similarly, $T\in\mathcal{C}$ implies $T\in\mathcal{D}$. Now pick $T\in\mathcal{D}$. Then $T=T_{1}\cup T_{2}\cup \{v\}$, where $\mathcal{N}_{G_{i}}(v)\not\subseteq T_{i}\in \mathfrak{C}(G_{i}\setminus\{v\})$ for $i=1,2$. Then $T_{i}\cup\{v\}\in \mathfrak{C}(\overline{G_{i}})$ and by \cite[Lemma 4.16]{acc} and Proposition \ref{accnoncut}, we have $v\neq t_{i}\in T_{i}$ such that $T_{i}\cup\{v\}\setminus\{t_{i}\}\in \mathfrak{C}(\overline{G_{i}})$ for $i=1,2$. Then $T\setminus\{t_{1}\}\in \mathcal{D}$ which implies $T$ is accessible. So, $\mathfrak{C}(G)$ is an accessible set system and unmixedness of $J_{G}$ is given. Hence $G$ is accessible.
\end{proof}

\begin{theorem}\label{baraccG}
Let $G$ be a simple connected graph such that $J_{G}$ is unmixed, and for each block $B$ of $G$, $\overline{B}$ is accessible. Then $G$ is accessible.
\end{theorem}

\begin{proof}
We will use induction on $n$, where $\vert V(G)\vert=n$. For, $n=1$ and $n=2$, $G$ is complete and we are done. If all the blocks of $G$ is complete then there is nothing to prove. Let $n>2$. If there is a block $B$ of $G$ for which $\overline{B}=G$, then $G$ is accessible as $\overline{B}$ is so. Now consider a block $B$ of $G$ which is not complete and take a cut vertex $v$ of $G$ belong to $V(B)$ such that $G\setminus\{v\}$ does not contain a $K_{1}$. Let $G=G_{1}\cup G_{2}$ be such that $V(G_{1})\cap V(G_{2})=\{v\}$ and $V(B)\subseteq V(G_{1})$. Now, Consider the graph $\overline{G_{i}}$ as described in Lemma \ref{Gbaracc}, where $i=1,2$. Then by Lemma \ref{barunm}, $J_{\overline{G_{1}}}$ and $J_{\overline{G_{2}}}$ are unmixed as $J_{G}$ is unmixed. Thus, by induction hypothesis, $\overline{G_{i}}$ is accessible for $i=1,2$. Hence by Lemma \ref{Gbaracc}, $G$ is accessible.
\end{proof}

\begin{lemma}\label{Gv-vsu}
Let $G=G_{1}\cup G_{2}$ be a simple connected graph such that $V(G_{1})\cap V(G_{2})=\{v\}$, where $v$ is a cut vertex of $G$. Let $J_{G}$ and $J_{G\setminus\{v\}}$ be unmixed. Then $J_{G_{v}}$ is strongly unmixed implies $J_{G_{v}\setminus\{v\}}$ is strongly unmixed.
\end{lemma}

\begin{proof}
By Proposition \ref{Gvvsu} and Proposition \ref{Gvsu}, we have $J_{G_{v}}$ is strongly unmixed if and only if  $J_{(G_{1})_{v}}$ and $J_{(G_{2})_{v}}$ are strongly unmixed. First we will show that if one of $\mathcal{N}_{G_{1}}(v)$ or $\mathcal{N}_{G_{2}}(v)$ is singleton then the result holds. Let $\mathcal{N}_{G_{1}}(v)=\{u\}$. Then $(G_{1})_{v}=G_{1}$. By Proposition \ref{accpropunm}, $J_{G}$ and $J_{G\setminus\{v\}}$ are unmixed imply $u\not\in T$ for any $T\in \mathfrak{C}(G_{1}\setminus\{v\})$ i.e., $u$ is a free vertex in $G_{1}\setminus\{v\}$. In this case, $G_{1}$ is a decomposable graph as 
$$G_{1}=(G_{1})_{v}=(G_{1}\setminus\{v\})\cup K_{2},$$
 where $V(G_{1}\setminus\{v\})\cap V(K_{2})=\{u\}$. Therefore $J_{G_{1}\setminus\{v\}}$ is strongly unmixed by Remark \ref{remglu}. Then we can see that 
$$G_{v}\setminus\{v\}=G_{1}\setminus\{v\}\cup (G_{2})^{\prime}_{v},$$
where $V(G_{1}\setminus\{v\})\cap V((G_{2})^{\prime}_{v})=\{u\}$ and $(G_{2})^{\prime}_{v}\simeq (G_{2})_{v}$, just relabeling the vertex $v$ of $(G_{2})_{v}$ by the vertex $u$. Again, using Remark \ref{remglu}, we get $J_{G_{v}\setminus\{v\}}$ is strongly unmixed.\par 
Now we will use induction on the number of vertices $n$ of $G_{v}\setminus\{v\}$. Since $v$ is a cut vertex of $G$, the base case is $n=2$ and in this case, $G_{v}\setminus\{v\}$ being a $K_{2}$, we are done. Let $n>2$ and the hypothesis is true for all such graphs with less than $n$ vertices. We can assume $\vert \mathcal{N}_{G_{i}}(v)\vert >1$ for $i=1,2$ and in this case, $(G_{i})_{v}$ is complete if and only if $(G_{i})_{v}\setminus\{v\}$ is complete as $J_{(G_{i})_{v}}$ is strongly unmixed, where $i=1,2$. Then we may assume one of $(G_{i})_{v}$ is not complete, say $(G_{1})_{v}$, otherwise, $G_{v}\setminus\{v\}$ will be complete and the result follows. Now there exists a cut vertex $x\in V(G_{1})$ for which the binomial edge ideals of $(G_{1})_{v}\setminus\{x\}, ((G_{1})_{v})_{x}$, and $((G_{1})_{v})_{x}\setminus\{x\}$ are strongly unmixed. Set $D=G_{v}\setminus\{v\}$. Let $(G_{1})_{v}\setminus\{x\}=H^{\prime}\sqcup H^{\prime\prime}$, where $v\in V(H^{\prime})$. Then $D\setminus \{x\}=H\sqcup H^{\prime\prime}$, where $V(G_{2}\setminus\{v\})\subseteq V(H)$. Note that $H=(G^{x})_{v}\setminus\{v\}$, where $G^{x}=(G_{2})_{v}\cup H^{\prime}$. Then $V((G_{2})_{v})\cap V(H^{\prime})=\{v\}$ and $v$ is a free vertex in both of $(G_{2})_{v}$ and $H^{\prime}$. Since $J_{(G_{2})_{v}}$ and $J_{H^{\prime}}$ are unmixed, by Remark \ref{remglu}, $J_{G^{x}}$ is unmixed. By Proposition \ref{accpropunm}, $J_{G_{2}\setminus\{v\}}$ is unmixed implies there exists no $S\in \mathfrak{C}(G_{2}\setminus\{v\})$ containing $\mathcal{N}_{G_{2}}(v)$ and so, $\mathcal{N}_{(G_{2})_{v}}(v)\not\subseteq T$ for any $T\in \mathfrak{C}((G_{2})_{v}\setminus\{v\})$. If there exists $T^{\prime}\in \mathfrak{C}(H^{\prime}\setminus\{v\})$ such that $\mathcal{N}_{H^{\prime}}(v)\subseteq T^{\prime}$, then $\mathcal{N}_{G_{1}}(v)\subseteq T^{\prime}\cup\{x\}\in \mathfrak{C}(G_{1}\setminus\{v\})$ which is a contradiction to the fact that $J_{G_{1}\setminus\{v\}}$ is unmixed. Therefore, by Proposition \ref{accpropunm}, $G^{x}\setminus\{v\}$ is unmixed. Now, $J_{H^{\prime}_{v}}=J_{H^{\prime}}$ and $J_{(G_{2})_{v}}$ are strongly unmixed and hence by Proposition \ref{Gvsu}, $J_{(G^{x})_{v}}$ is strongly unmixed. Therefore, by induction hypothesis we have $J_{H}$ is strongly unmixed and so is $J_{D\setminus\{x\}}$. Also, we can observe that $$D_{x}\setminus\{x\}= L_{v}\setminus\{v\}\,\,\text{with}\,\, L=\big(((G_{1})_{v})_{x}\setminus\{x\}\big)\cup (G_{2})_{v}.$$
Then $L$ is a decomposable graph and so, $J_{L}$ is unmixed as $J_{((G_{1})_{v})_{x}\setminus\{x\}}$ and $J_{(G_{2})_{v}}$ are unmixed. We have already proved the binomial edge ideals of $(G_{2})_{v}\setminus\{v\}$ and $((G_{1})_{v}\setminus\{v\})\setminus\{x\}$ are unmixed. Therefore, the binomial edge ideal of $\big(((G_{1})_{v})_{x}\setminus\{x\}\big)\setminus\{v\}$ is unmixed and so is $J_{L\setminus\{v\}}$. Again using Proposition \ref{Gvsu}, we see that $J_{L_{v}}$ is strongly unmixed and thus, by induction, $J_{D_{x}\setminus\{x\}}$ is strongly unmixed. By, Lemma \ref{freeGsu}, $J_{D_{x}}$ is also strongly unmixed. Hence $J_{G_{v}\setminus\{v\}}$ is strongly unmixed.
\end{proof}

Recall the notion of block graph of a given graph. Let $G$ be a graph with blocks $B_{1},\ldots, B_{m}$. We denote the \textit{block graph} of $G$ by $\mathscr{B}(G)$, which is defined as follows:
\begin{enumerate}
\item[$\bullet$] $V(\mathscr{B}(G))=\{B_{1},\ldots,B_{m}\}$.
\item[$\bullet$] $E(\mathscr{B}(G))=\{\{B_{i},B_{j}\}\mid V(B_{i})\cap V(B_{j})\neq \phi\}$.
\end{enumerate}
\medskip

\begin{theorem}\label{barsuG}
Let $G$ be a simple graph such that $J_{G}$ is unmixed. If $J_{\overline{B}}$ is strongly unmixed for each block $B$ of $G$, then $J_{G}$ is strongly unmixed.
\end{theorem}

\begin{proof}
If $G$ has no cut vertex, then the result follows trivially.
So assume $G$ has some cut vertices. First we will prove the following claim:\\
\textbf{Claim:} There exist two blocks $B_{p}$ and $B_{q}$ of $G$ such that $V(B_{p})\cap V(B_{q})=\{v\}$ and the binomial edge ideals of $\overline{B_{i}}\setminus\{v\}, (\overline{B_{i}})_{v}$, and $(\overline{B_{i}})_{v}\setminus\{v\}$ are strongly unmixed, where $i\in\{p,q\}$.\\
\textit{Proof of claim:} Since $J_{G}$ is unmixed, by \cite[Proposition 1.3]{cactus}, $\mathscr{B}(G)$ is a tree and this implies each cut vertex of $G$ lies exactly in two blocks of $G$. Take a block $B_{1}$ which contains only one cut vertex $v_{1}$ of $G$. Then $J_{\overline{B_{1}}}$ is strongly unmixed gives $J_{\overline{B_{1}}\setminus\{v_{1}\}}, J_{(\overline{B_{1}})_{v_{1}}}$, and $J_{(\overline{B_{1}})_{v_{1}}\setminus\{v_{1}\}}$ are strongly unmixed. Let $B_{2}$ be the other block of $G$ containing $v_{1}$. Now, $J_{\overline{B_{2}}}$ is strongly unmixed and so, there exists a cut vertex $v_{2}$ of $\overline{B_{2}}$ for which $J_{\overline{B_{2}}\setminus\{v_{2}\}}, J_{(\overline{B_{2}})_{v_{2}}}$, and $J_{(\overline{B_{2}})_{v_{2}}\setminus\{v_{2}\}}$ are strongly unmixed. If $v_{1}=v_{2}$, then we are done. Otherwise, consider the other block $B_{3}$ which contains the cut vertex $v_{2}$. Continuing this process, if we do not get a cut vertex $v_{j}$ for which the binomial edge ideals of $\overline{B_{j}}\setminus\{v_{j}\}, (\overline{B_{j}})_{v_{j}}, (\overline{B_{j}})_{v_{j}}\setminus\{v_{j}\}$ and $\overline{B_{j+1}}\setminus\{v_{j}\}, (\overline{B_{j+1}})_{v_{j}}, (\overline{B_{j+1}})_{v_{j}}\setminus\{v_{j}\}$ are strongly unmixed, then after finite number of step we will reach to a block $B_{k}$ which has only one cut vertex. Therefore, by the given conditions, the binomial edge ideals of $\overline{B_{k}}\setminus\{v_{k-1}\}, (\overline{B_{k}})_{v_{k-1}}, (\overline{B_{k}})_{v_{k-1}}\setminus\{v_{k-1}\}$ are strongly unmixed and also, the binomial edge ideals of $\overline{B_{k-1}}\setminus\{v_{k-1}\}, (\overline{B_{k-1}})_{v_{k-1}}, (\overline{B_{k-1}})_{v_{k-1}}\setminus\{v_{k-1}\}$ are strongly unmixed. Thus, in $G$ we will get two blocks $B_{p}$ and $B_{q}$ such that $V(B_{p})\cap V(B_{q})=\{v\}$ and the binomial edge ideals of $\overline{B_{i}}\setminus\{v\}, (\overline{B_{i}})_{v}$, and $(\overline{B_{i}})_{v}\setminus\{v\}$ are strongly unmixed, where $i\in\{p,q\}$.\par 
Due to Remark \ref{remglu}, we may assume $G$ is indecomposable. We proceed by induction on the number of vertices $n$ of $G$. Let $G\setminus\{v\}=H_{1}\sqcup H_{2}$, where $V(B_{p}\setminus\{v\})\subseteq V(H_{1})$ and $V(B_{q}\setminus\{v\})\subseteq V(H_{2})$. Let $B$ be a block of $H_{1}$. Suppose $B_{p}$ is $K_{2}$, then we can write $G=G_{1}\cup B_{p}\cup G_{2}$ such that $V(G_{1})\cap V(B_{p})=\{v\}$ and $V(G_{2})\cap V(B_{p})=\{w\}$. Since $G$ is indecomposable, $v$ and $w$ can not be free vertices in $G_{1}$ and $G_{2}$, respectively. So, there exists $T_{i}\in \mathfrak{C}(G_{i})$ for $i=1,2$ such that $v\in T_{1}$ and $w\in T_{2}$. Note that $T_{1},T_{2},T_{1}\cup T_{2}\in \mathfrak{C}(G)$. Since $J_{G}$ is unmixed, $c_{G_{i}}(T_{i})=\vert T_{i}\vert$ for $i=1,2$. Therefore, $c_{G}(T_{1}\cup T_{2})=\vert T_{1}\cup T_{2}\vert$ which gives a contradiction as $J_{G}$ is unmixed. Hence $B_{p}$ can not be $K_{2}$ and similarly, $B_{q}$ can not be $K_{2}$. In this situation, if $B$ is also a block of $G$, then $\overline{B}$ with respect to $H_{1}$ and $\overline{B}$ with respect to $G$ are same. So, $J_{\overline{B}}$ is strongly unmixed. If $B$ is not a block of $G$, then $V(B)\subseteq V(B_{p}\setminus\{v\})$ and in this case, $\overline{B}$ with respect to $H_{1}$ is equal to $\overline{B}$ with respect to $\overline{B_{p}}\setminus\{v\}$. Since $J_{\overline{B_{p}}\setminus\{v\}}$ is strongly unmixed, by Theorem \ref{blocksu}, $J_{\overline{B}}$ is strongly unmixed. Also, $J_{H_{1}}$ is unmixed by Proposition \ref{accpropunm}. Therefore, by induction hypothesis, $J_{H_{1}}$ is strongly unmixed. Similarly, $J_{H_{2}}$ is strongly unmixed and so is $J_{G\setminus\{v\}}$. For a fix $n$, we will use induction on the number of blocks $k$ of $G$. If $k=1$, then by the given condition $G$ is complete and so $J_{G}$ is strongly unmixed. Now, consider $G_{v}$ and $G_{v}\setminus\{v\}$. Then for any block $B$ of $G_{v}$ and $G_{v}\setminus\{v\}$ such that $V(B_{p}\cup B_{q})\setminus\{v\}\not\subseteq B$, we have $\overline{B}$ with respect to $G_{v}$ or $G_{v}\setminus\{v\}$ is same as $\overline{B}$ with respect to $G$. Thus, $J_{\overline{B}}$ is strongly unmixed. Let $D=B_{p}\cup B_{q}$. Then $D_{v}$ is the block of $G_{v}$ containing $V(B_{p}\cup B_{q})$ and $D_{v}\setminus\{v\}$ is the block of $G_{v}\setminus\{v\}$ containing $V(B_{p}\cup B_{q})\setminus\{v\}$. Since the binomial edge ideals of $(\overline{B_{p}})_{v}\setminus\{v\}$ and $(\overline{B_{q}})_{v}\setminus\{v\}$ are strongly unmixed, by Proposition \ref{Gvsu}, $J_{\overline{D_{v}}}$ is strongly unmixed and by Lemma \ref{Gv-vsu}, $J_{\overline{D_{v}}\setminus\{v\}}$ is strongly unmixed. Number of blocks in $G_{v}$ and $G_{v}\setminus\{v\}$ is $k-1$. Thus, by induction hypothesis, we have $J_{G_{v}}$ and $J_{G_{v}\setminus\{v\}}$ are strongly unmixed and hence, $J_{G}$ is strongly unmixed. 
\end{proof}

In \cite[Problem 7.2]{acc}, the authors proposed the following question:

\begin{question}\label{quesidentify}{\rm
Let $G$ and $H$ be two disjoint connected graphs such that $J_{G}$ and $J_{H}$ are unmixed. Let $v,w$ be the cut vertices of $G, H$, respectively, for which $J_{G\setminus\{v\}}$ and $J_{H\setminus\{w\}}$ are unmixed. Set $G\setminus\{v\} = G_{1}\sqcup G_{2}, H\setminus\{w\} = H_{1}\sqcup H_{2}$. Let $F_{ij}$ be the graph obtained by gluing $G[V(G_{i})\cup \{v\}]$ and $H[V(H_{j})\cup \{w\}]$ identifying $v$ and $w$, where $i,j = 1,2$. If $J_{G}$ and $J_{H}$ are Cohen-Macaulay, is it true that $J_{F_{ij}}$ is Cohen-Macaulay$?$ If $G$ and $H$ are accessible, is it
true that $F_{ij}$ is accessible$?$
}
\end{question}

Using some previous results of this section, we manage to give a partial answer to Question \ref{quesidentify} through the following Theorem \ref{identifythm}.

\begin{theorem}\label{identifythm}
Let $G=G_{1}\cup G_{2}$ with $V(G_{1})\cap V(G_{2})=\{v\}$ and $H=H_{1}\cup H_{2}$ with $V(H_{1})\cap V(H_{2})=\{w\}$ be two distinct connected graphs such that $J_{G},J_{H}, J_{G\setminus\{v\}},J_{H\setminus\{w\}}$ are unmixed. Consider the graph $F_{ij}=G_{i}\cup H_{j}$ identifying the vertices $v$ and $w$, labeling as $v$ i.e., $V(G_{i})\cap V(H_{j})=\{v\}$, for $i,j\in\{1,2\}$. Then the following hold for each $i,j\in\{1,2\}$.
\begin{enumerate}[(i)]
\item If $G$ and $H$ are accessible, then $F_{ij}$ is accessible.
\item If $J_{G}$ and $J_{H}$ are strongly unmixed, then $J_{F_{ij}}$ is strongly unmixed.
\end{enumerate}
\end{theorem}

\begin{proof}
Choose an $F_{ij}$ for $i,j\in\{1,2\}$. Since $J_{G\setminus\{v\}}$ and $J_{H\setminus\{v\}}$ are unmixed, $J_{F_{ij}\setminus\{v\}}$ is unmixed. Let $S_{i}\in \mathfrak{C}(G_{i})$ with $v\not\in S_{i}$. Then $S_{i}\in \mathfrak{C}(G)$ and $c_{G}(S_{i})=c_{G_{i}}(S_{i})=\vert S_{i}\vert +1$ as $J_{G}$ is unmixed. Suppose $S_{i}\in \mathfrak{C}(G_{i})$ with $v\in S_{i}$. Then $S_{i}\in\mathfrak{C}(G)$ and $c_{G}(S_{i})=\vert S_{i}\vert +1$ as $J_{G}$ is unmixed. In this case, one connected component of $G\setminus S_{i}$ is $G_{k}\setminus\{v\}$, where $i\neq k\in\{1,2\}$. Therefore, $c_{G_{i}}(S_{i})=\vert S_{i}\vert$. Thus, $G_{i}$ satisfy the condition (ii) of Corollary \ref{unmixed}. Similarly, $F_{j}$ also satisfy the condition (ii) of Corollary \ref{unmixed}. Hence by Corollary \ref{unmixed}, $J_{F_{ij}}$ is unmixed. Since, $G$ and $H$ are accessible, by Theorem \ref{baracc} and Theorem \ref{blocksu}, for each block $B$ of $G$, $\overline{B}$ with respect to $G$ is accessible and strongly unmixed. Same holds for $H$. Note that for any block $B$ of $F_{ij}$, $\overline{B}$ with respect to $F_{ij}$ is either $\overline{B}$ with respect to $G$ or $\overline{B}$ with respect to $H$. Therefore, by Theorem \ref{baraccG}, $F_{ij}$ is accessible and by Theorem \ref{barsuG}, $J_{F_{ij}}$ is strongly unmixed.
\end{proof}

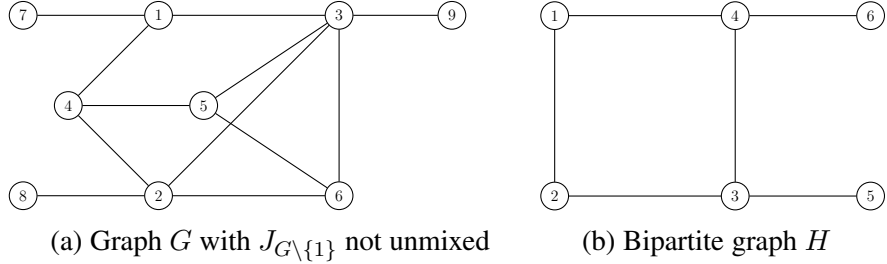
\begin{figure}[H]
	\centering
	\begin{subfigure}{0.55\textwidth}
	
	\begin{tikzpicture}
  [scale=.6,auto=left,every node/.style={circle,scale=0.5}]
 
  \node[draw] (n1) at (0,4)  {$1$};
  \node[draw] (n2) at (0,0)  {$2$};
  \node[draw] (n3) at (4,4) {$3$};
   \node[draw] (n4) at (-2,2) {$4$};
   \node[draw] (n5) at (1,2) {$5$};
  \node[draw] (n6) at (4,0) {$6$};
  \node[draw] (n7) at (-3,4) {$7$};
  \node[draw] (n8) at (-3,0) {$8$};
  \node[draw] (n9) at (6.5,4) {$9$};
 
  \foreach \from/\to in {n1/n3,n1/n4,n1/n7, n2/n3, n2/n4, n2/n6, n2/n8, n3/n5, n3/n6, n3/n9, n4/n5, n5/n6}
    \draw[] (\from) -- (\to);
    
\end{tikzpicture}
\caption{Graph $G$ with $J_{G\setminus\{1\}}$ not unmixed}\label{figsu1}
	\end{subfigure}
%%%%%%%%%%%%%%
	\begin{subfigure}{0.35\textwidth}
	
	\begin{tikzpicture}
  [scale=.6,auto=left,every node/.style={circle,scale=0.5}]
    
  \node[draw] (n1) at (0,4)  {$1$};
  \node[draw] (n2) at (0,0)  {$2$};
  \node[draw] (n3) at (4,0) {$3$};
   \node[draw] (n4) at (4,4) {$4$};
   \node[draw] (n5) at (7,0) {$5$};
  \node[draw] (n6) at (7,4) {$6$};

  \foreach \from/\to in {n1/n2,n1/n4, n2/n3, n3/n4, n3/n5, n4/n6}
    \draw[] (\from) -- (\to);
   
\end{tikzpicture}
\caption{Bipartite graph $H$}\label{figsu2}
	\end{subfigure}
	
	\caption{$J_{G}$ and $J_{H}$ are Cohen-Macaulay.}\label{figsu}
\end{figure}
\medskip

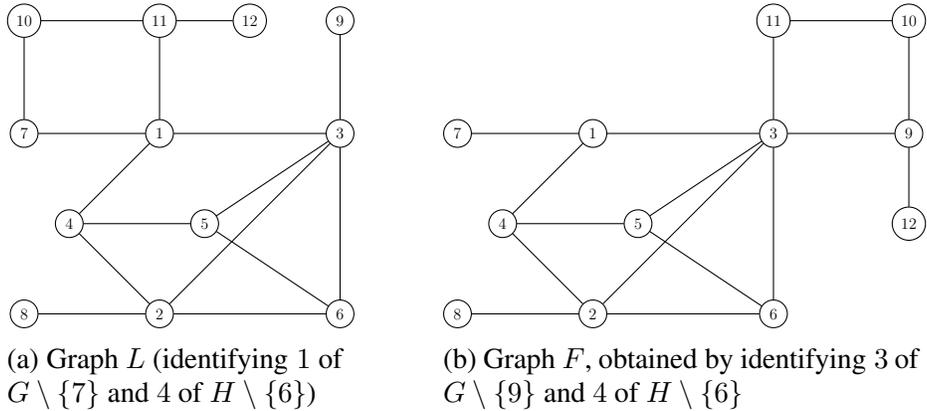
\begin{figure}[H]
	\centering
	\begin{subfigure}{0.45\textwidth}
	\begin{tikzpicture}
  [scale=.6,auto=left,every node/.style={circle,scale=0.5}]
    
  \node[draw] (n1) at (0,4)  {$1$};
  \node[draw] (n2) at (0,0)  {$2$};
  \node[draw] (n3) at (4,4) {$3$};
   \node[draw] (n4) at (-2,2) {$4$};
   \node[draw] (n5) at (1,2) {$5$};
  \node[draw] (n6) at (4,0) {$6$};
  \node[draw] (n7) at (-3,4) {$7$};
  \node[draw] (n8) at (-3,0) {$8$};
  \node[draw] (n9) at (4,6.5) {$9$};
  \node[draw] (n10) at (-3,6.5) {$10$};
  \node[draw] (n11) at (0,6.5) {$11$};
  \node[draw] (n12) at (2,6.5) {$12$};

  \foreach \from/\to in {n1/n7,n1/n3, n1/n4, n2/n3, n2/n4, n2/n6, n2/n8, n3/n5, n3/n6, n3/n9, n4/n5, n5/n6, n7/n10, n10/n11, n1/n11, n11/n12}
    \draw[] (\from) -- (\to);
   
\end{tikzpicture}
\caption{Graph $L$ (identifying $1$ of\\
$G\setminus\{7\}$ and $4$ of $H\setminus\{6\}$)}\label{fignonsuglu}
	\end{subfigure}
%%%%%%%%%%%%%%
\begin{subfigure}{0.5\textwidth}
	\begin{tikzpicture}
  [scale=.6,auto=left,every node/.style={circle,scale=0.5}]
    
  \node[draw] (n1) at (0,4)  {$1$};
  \node[draw] (n2) at (0,0)  {$2$};
  \node[draw] (n3) at (4,4) {$3$};
   \node[draw] (n4) at (-2,2) {$4$};
   \node[draw] (n5) at (1,2) {$5$};
  \node[draw] (n6) at (4,0) {$6$};
  \node[draw] (n7) at (-3,4) {$7$};
  \node[draw] (n8) at (-3,0) {$8$};
  \node[draw] (n9) at (7,4) {$9$};
  \node[draw] (n10) at (7,6.5) {$10$};
  \node[draw] (n11) at (4,6.5) {$11$};
  \node[draw] (n12) at (7,2) {$12$};

  \foreach \from/\to in {n1/n7,n1/n3, n1/n4, n2/n3, n2/n4, n2/n6, n2/n8, n3/n5, n3/n6, n3/n9, n4/n5, n5/n6, n9/n10, n10/n11, n3/n11, n9/n12}
    \draw[] (\from) -- (\to);
   
\end{tikzpicture}
\caption{Graph $F$, obtained by identifying $3$ of $G\setminus\{9\}$ and $4$ of $H\setminus\{6\}$}\label{figsuglu}
	\end{subfigure}

	\caption{$J_{F}$ is Cohen-Macaulay but $J_{L}$ is not.}\label{figglu}
\end{figure}

\begin{example}{\rm
Consider the graphs in Figure \ref{figsu} and Figure \ref{figglu}. Binomial edge ideal of the graph $G$ in \Cref{figsu1} is strongly unmixed (proved in \cite[Example 6.11]{acc}). The graph $H$ in \Cref{figsu2} is bipartite accessible graph and so, $J_{H}$ is strongly unmixed by \cite[Corollary 6.9]{acc}. In \Cref{fignonsuglu}, $L$ is the graph obtained by attaching $G\setminus\{7\}$ and $H\setminus\{6\}$ identifying the vertex $1$ of $G\setminus\{7\}$ and the vertex $4$ of $H\setminus\{6\}$ with some relabeling of vertices. Similarly, in \Cref{figsuglu}, $F$ is the graph obtained by gluing $G\setminus\{9\}$ and $H\setminus\{6\}$ at the vertices $3$ of $G\setminus\{9\}$ and the vertex $4$ of $H\setminus\{6\}$ with some relabeling of vertices. \par 
Note that $G\setminus\{1,7\}$ is not unmixed as $J_{G}$ is unmixed and $\{3,4,6\}\in \mathfrak{C}(G\setminus\{1\})$. Therefore, by  Proposition \ref{cutset}, $\{3,4,6,1,10\}\in \mathfrak{C}(L)$ and $c_{L}(\{3,4,6,1,10\})=5\neq 5+1$. Thus, $J_{L}$ is not unmixed and so $J_{L}$ can not be Cohen-Macaulay. On the other hand, $G\setminus\{3,9\}$ and $H\setminus\{4,6\}$ is unmixed by Proposition \ref{accpropunm}. Therefore, by Corollary \ref{unmixed}, $J_{F}$ is unmixed and hence, by Theorem \ref{barsuG}, $J_{F}$ is strongly unmixed as well as Cohen-Macaulay.
}
\end{example}

\section{Cohen-Macaulay Binomial Edge Ideals of $r$-Regular $r$-Connected Blocks with Whiskers }\label{secstarsu}
\medskip

In this part, we define a new class of indecomposable graphs by attaching two complete graphs in a particular manner and adding some whiskers. We show the Conjecture \ref{accsuconj} holds for these graphs and classify all $r$-regular $r$-connected graphs with whiskers whose binomial edge ideals are Cohen-Macaulay.
\medskip

\noindent\textbf{Construction:} Take two complete graphs $K_{m}$ and $K_{n}$ with $V(K_{m})=\{x_{1},\ldots, x_{m}\}$ and $V(K_{n})=\{y_{1},\ldots, y_{n}\}$. Let $r\leq m$ and $r\leq n$ be a positive integer. The \textit{star product} of $K_{m}$ and $K_{n}$ with respect to $r$, denoted by $K_{m}\star_{r} K_{n}$, is defined in such a way that
\begin{enumerate}
\item[$\bullet$] $V(K_{m}\star_{r} K_{n})=V(K_{m})\sqcup V(K_{n}).$
\item[$\bullet$] $E(K_{m}\star_{r} K_{n})=E(K_{m})\cup E(K_{n})\cup\{\{x_{i},y_{i}\}\mid 1\leq i\leq r\}$.
\end{enumerate}

Note that $K_{m}\star_{r} K_{n}=K_{n}\star_{r} K_{m}$ i.e., the star product is commutative. For $m=n=r$, we write $K_{r}\star K_{r}$ instead of $K_{r}\star_{r} K_{r}$.
\medskip

Now we will consider the graph $\overline{K_{m}\star_{r} K_{n}}$, by  adding some whiskers to it in a special manner such that

\begin{enumerate}
\item[$\bullet$] $V(\overline{K_{m}\star_{r} K_{n}})=V(K_{m}\star_{r} K_{n})\cup \{f_{x_{i}},f_{y_{i}}\mid 2\leq i\leq r\};$
\item[$\bullet$] $E(\overline{K_{m}\star_{r} K_{n}})=E(K_{m}\star_{r} K_{n})\cup\{\{x_{i},f_{x_{i}}\}, \{y_{i},f_{y_{i}}\}\mid 2\leq i\leq r\}$.
\end{enumerate}
\medskip

\begin{lemma}\label{starunm}
Let $G=K_{m}\star_{r} K_{n}$. Then $J_{\overline{G}}$ is unmixed.
\end{lemma}  

\begin{proof}
We will first prove $G_{r}=\overline{K_{r}\star K_{r}}$ is unmixed proceeding by induction on $r$.\par
For $r=1$, $G_{1}$ is $K_{2}$ and $\phi$ being the only cutset of $G_{1}$, $J_{G_{1}}$ is unmixed. When $r=2$, we can see that $J_{G_{2}}$ is unmixed by \cite[Lemma 2.9]{cactus}. Now consider the graph $G_{r}\setminus \{x_{r}\}=\{f_{x_{r}}\}\sqcup H$. Clearly, $H$ is a decomposable graph such that $H= H_{1}\cup H_{2}$ with $H_{1}=H[\{y_{r},f_{y_{r}}\}]$ and $H_{2}=H\setminus \{f_{y_{r}}\}$. Notice that $G_{r-1}=H_{2}\setminus \{y_{r}\}$. Since $y_{r}$ is a free vertex in $H_{2}$, $y_{r}\not\in S$ for all $S\in \mathfrak{C}(H_{2})$. Also, $\mathcal{N}_{H_{2}}(y_{r})\not \subseteq T$ for any $T\in \mathfrak{C}(G_{r-1})$. By induction hypothesis, $J_{G_{r-1}}$ is unmixed and therefore, by Corollary \ref{unmixed} and \cite[Lemma 2.2]{raufrin}, we have $J_{H}$ is unmixed. Again, $G_{r}\setminus \{f_{x_{r}}\}$ and $G_{r}[\{x_{r},f_{x_{r}}\}]$ satisfy the condition (ii) of Corollary \ref{unmixed}. Hence using Corollary \ref{unmixed}, we get $J_{G_{r}}$ is unmixed as $J_{G_{r}\setminus\{x_{r}\}}$ is unmixed. \par 
If $m=r+1, n=r$ or $m=r, n=r+1$ or $m=r+1, n=r+1$, then we can observe that $c_{\overline{G}}(T)=\vert T\vert+1$ for all $T\in \mathfrak{C}(\overline{G})$ with $\mathcal{N}_{\overline{G}}(x_{r+1})\subseteq T$ or $\mathcal{N}_{\overline{G}}(y_{r+1})\subseteq T$. Therefore, using Lemma \ref{freunm} we get $J_{\overline{G}}$ is unmixed. For $m>r+1$ or $n>r+1$,  $x_{r+1},\ldots,x_{m}$ and $y_{r+1},\ldots, y_{n}$ are free vertex of $\overline{G}$. Note that $\mathcal{N}_{\overline{G}}(x_{i})\not\in \mathfrak{C}(\overline{G})$ for all $r+1\leq i\leq m$ if $m>r+1$ and $\mathcal{N}_{\overline{G}}(y_{i})\not\in \mathfrak{C}(\overline{G})$ for all $r+1\leq i\leq n$ if $n>r+1$ . Therefore, by repeating application of Lemma \ref{freunm}, we get $J_{\overline{G}}$ is unmixed.
\end{proof}

\begin{lemma}\label{staracc}
Let $G=K_{m}\star_{r} K_{n}$. Then $\overline{G}$ is accessible.
\end{lemma}

\begin{proof}
We first prove $G_{r}=\overline{K_{r}\star K_{r}}$ is accessible by induction on $r$.\par

For $r=1$, $G_{1}=K_{2}$ and by \cite[Remark 4.2]{acc}, it is accessible. For $r=2$, $J_{G_{2}}$ is Cohen-Macaulay by \cite[Lemma 2.9]{cactus} and hence $G_{2}$ is accessible by \cite[Theorem 3.5]{acc}. Let us consider the graph $G_{r}\setminus \{x_{r}\}=\{f_{x_{r}}\}\sqcup H$. Now $H$ is a decomposable graph with the decomposition $H=H_{1}\cup H_{2}$, where $H_{1}=H[\{y_{r},f_{y_{r}}\}]$ and $H_{2}=H\setminus \{f_{y_{r}}\}$. $y_{r}$ is a free vertex of $H_{2}$ and $H_{2}\setminus \{y_{r}\}=G_{r-1}$, which is accessible by induction hypothesis. Also, unmixedness of $J_{H_{2}}$ is clear from proof of Lemma \ref{starunm}. Therefore, by Lemma \ref{freunm}, $H_{2}$ is accessible. $H_{1}$ being $K_{2}$ is accessible. Thus, by Remark \ref{remglu}, $H$ is accessible. Now $G_{r}=G^{1}\cup G^{2}$, where $V(G^{1})\cap V(G^{2})=\{x_{r}\}$. Then $G^{1}$ is $K_{2}$ with $G^{1}\setminus \{x_{r}\}=\{f_{x_{r}}\}$ and $G^{2}\setminus \{x_{r}\}=H$. Clearly, $\mathfrak{C}(G^{1})=\mathfrak{C}(G^{1}\setminus\{x_{r}\})=\{\phi\}$. Considering $G_{r}=G^{1}\cup G^{2}$, we have $\mathfrak{C}(G_{r})=\mathcal{A}\cup\mathcal{B}\cup\mathcal{C}\cup\mathcal{D}$ as described in Proposition \ref{cutset}. Let $T\in \mathcal{A}$. Then $T=S_{2}$, where $x_{r}\not\in S_{2}\in \mathfrak{C}(G^{2})$. If $x_{i}\in \{x_{1},\ldots,x_{r-1}\}\cap S_{2}$, then $S_{2}\setminus \{x_{i}\}\in \mathfrak{C}(H)$ and so, $S_{2}$ is accessible. If $\{x_{1},\ldots,x_{r-1}\}\cap S_{2}=\phi$, then $S_{2}\in \mathfrak{C}(H)$ and so is accessible. Let $T\in \mathcal{B}$. Then $T=\phi\cup S_{2}=S_{2}$, where $x_{r}\in S_{2}\in \mathfrak{C}(G^{2})$. Then from the construction of $G_{r}$ it is easy to observe that $T\setminus\{x_{r}\}\in \mathfrak{C}(H)$ and hence $T$ is accessible. Since $H$ is accessible and $\mathfrak{C}(G^{1})=\mathfrak{C}(G^{1}\setminus\{x_{r}\})=\phi$, any $T\in\mathcal{C}\cup\mathcal{D}$ is accessible. Thus $G_{r}$ is accessible. Now for $m>r$ or $n>r$, $x_{r+1},\ldots,x_{m}$ and $y_{r+1},\ldots,y_{n}$ are free vertex of $\overline{G}$. From Lemma \ref{starunm}, $J_{\overline{G}}$ is unmixed and hence by repeating application of Proposition \ref{freacc} we get $\overline{G}$ is accessible.
\end{proof}

\begin{theorem}\label{starsu}
Let $G=K_{m}\star_{r} K_{n}$. Then $J_{\overline{G}}$ is strongly unmixed. 
\end{theorem}

\begin{proof}
Let $G_{r}=\overline{K_{r}\star K_{r}}$. For $m>r$ or $n>r$, $x_{r+1},\ldots,x_{m}$ and $y_{r+1},\ldots,y_{n}$ are free vertex of $\overline{G}$. Using repeating applications of Lemma \ref{starunm} and Lemma \ref{freeGsu}, it is enough to prove $J_{G_{r}}$ is strongly unmixed to show $J_{\overline{G}}$ is strongly unmixed.\par

Note that $G_{1}=K_{2}$ and so is $J_{G_{1}}$ is strongly unmixed by definition. Also, $G_{2}$ is a bipartite accessible graph and hence $J_{G_{2}}$ is strongly unmixed by \cite[Corollary 6.9]{acc}. We proceed by induction on $r$. By Lemma \ref{starunm}, $J_{G_{r}}$ is unmixed. Consider $G_{r}\setminus\{x_{r}\}=\{f_{x_{r}}\}\sqcup H$. Then by proof of Lemma \ref{staracc}, we have $G_{r}$ is accessible and $H$ is accessible. Therefore by \cite[Corollary 5.16]{acc}, $G_{r}\setminus \{x_{r}\}, (G_{r})_{x_{r}},\,\,\text{and}\,\, (G_{r})_{x_{r}}\setminus\{x_{r}\}$ are accessible and so, the corresponding binomial edge ideals are unmixed. Now $H$ is decomposable as $H=H_{1}\cup H_{2}$, where $H_{1}=H[\{y_{r},f_{y_{r}}\}]$ and $H_{2}=H\setminus \{f_{y_{r}}\}$. Note that $H_{2}\setminus \{y_{r}\}=G_{r-1}$ and so $J_{H_{2}\setminus\{y_{r}\}}$ is strongly unmixed by induction hypothesis. Since $y_{r}$ is a free vertex of $H_{2}$ and $J_{H_{2}}$ is unmixed, by Lemma \ref{freeGsu}, $J_{H_{2}}$ is strongly unmixed. Thus, $J_{H}$ is strongly unmixed by Remark \ref{remglu} and so is $J_{G_{r}\setminus\{x_{r}\}}$. Now consider the graph $(G_{r})_{x_{r}}\setminus \{x_{r}\}=D$ and set $D^{\prime}=D\setminus\{f_{x_{r}}\}$. Observe that $D^{\prime}\setminus\{y_{r}\}=\{f_{y_{r}}\}\sqcup G_{r-1}$. By induction hypothesis $J_{G_{r-1}}$ is strongly unmixed and so is $J_{D^{\prime}\setminus \{y_{r}\}}$. Also, $D^{\prime}_{y_{r}}$ and $D^{\prime}_{y_{r}}\setminus \{y_{r}\}$ are complete graph with some whiskers and hence $J_{D^{\prime}_{y_{r}}}$ and $J_{D^{\prime}_{y_{r}}\setminus\{y_{r}\}}$ are strongly unmixed. Now 
$$\mathcal{N}_{D}(f_{x_{r}})=\{x_{1},\ldots,x_{r-1},y_{r}\}\not\in \mathfrak{C}(D)$$
 as $x_{1}$ can not be a cut point in $D\setminus \{x_{2},\ldots,x_{r-1},y_{r}\}$. Since $f_{x_{r}}$ is a free vertex in $D$, from Lemma \ref{freunm} we have $J_{D^{\prime}}$ is unmixed. Therefore by definition $J_{D^{\prime}}$ is strongly unmixed. Using Lemma \ref{freeGsu} we get $J_{D}$ is strongly unmixed. $(G_{r})_{x_{r}}$ is unmixed, $x_{r}$ is a free vertex of $(G_{r})_{x_{r}}$ and $J_{D}$ is strongly unmixed together imply $J_{(G_{r})_{x_{r}}}$ is strongly unmixed (using Lemma \ref{freeGsu}). Hence by definition $J_{G_{r}}$ is strongly unmixed.
\end{proof}

\begin{lemma}\label{r-regcomp}
Let $B$ be a non-complete block such that $B$ is $r$-regular and $r$-connected. If $\overline{B}$ is accessible, then for any $v\in V(B)$, $B\setminus \mathcal{N}_{B}(v)$ contains two connected components.
\end{lemma}

\begin{proof}
Let $v\in V(B)$ be such that $B\setminus \mathcal{N}_{B}(v)$ contains the connected components $A_{1},\ldots, A_{s}$ and $\{v\}$. Since $B$ is non-complete, $r$-regular and $r$-connected, $\mathcal{N}_{B}(v)\in \mathfrak{C}(\overline{B})$ is clear. By accessibility of $\overline{B}$, there exists $w\in \mathcal{N}_{B}(v)$ such that $\mathcal{N}_{B}(v)\setminus\{w\}\in \mathfrak{C}(\overline{B})$. Now any $A_{i}$ is adjacent to all vertices belong to $ \mathcal{N}_{B}(v)$, otherwise we will get less than $r$ vertices to disconnect the block $B$. Since $\mathcal{N}_{B}(v)\setminus \{w\}\in \mathfrak{C}(\overline{B})$, by \cite[Lemma 4.14]{acc}, $w$ is adjacent to exactly two connected components of $\overline{B}\setminus \mathcal{N}_{B}(v)$ one of which contains $v$. Therefore $s$ must be $1$ and $w$ can not be a cut vertex of $\overline{B}$ or $s=0$ and $w$ is a cut vertex of $\overline{B}$. But, $s=0$ implies $B$ is complete and so $s=1$ is the only possibility.
\end{proof}

\begin{proposition}\label{r-regacc}
Let $B$ be a non-complete $r$-regular $r$-connected block. If $\overline{B}$ is accessible, then $\overline{B}=G_{r}=\overline{K_{r}\star K_{r}}$, where $r\geq 2$.
\end{proposition}

\begin{proof}
For each $v\in V(B)$, we have $\mathcal{N}_{B}(v)\in \mathfrak{C}(\overline{B})$ as $B$ is non-complete $r$-regular $r$-connected block. Unmixed property of $J_{\overline{B}}$ implies $$c_{\overline{B}}(\mathcal{N}_{B}(v))=\vert \mathcal{N}_{B}(v)\vert +1=r+1.$$
From Lemma \ref{r-regcomp} it follows that $B\setminus \mathcal{N}_{B}(v)$ contains two connected components and therefore $\mathcal{N}_{B}(v)$ should have $r-1$ cut vertices of $\overline{B}$. Let $L$ be the set of cut vertices of $\overline{B}$ and $\vert V(B)\vert=n$. Since each vertex of $B$ add $r-1$ degree to cut vertices, we have
\begin{align*}
\sum_{v\in L} \mathrm{deg}_{B}(v)=r\vert L\vert = (r-1)n\,\, \Rightarrow\,\, \vert L\vert =\dfrac{(r-1)n}{r}.
\end{align*}
Therefore, the number of non-cut vertices of $\overline{B}$ in $B$ is $n-\vert L\vert=\dfrac{n}{r}$. Since each cut vertex of $\overline{B}$ is adjacent to a non-cut vertex in $B$, we have $L\in \mathfrak{C}(\overline{B})$. $J_{\overline{B}}$ being unmixed, $c_{\overline{B}}(L)=\vert L\vert +1$ and so, the induced subgraph on the set of non-cut vertices of $B$ is connected. But, the degree of any non-cut vertex in the induced subgraph $B\setminus L$ will be one. So the only possibility is $B\setminus L$ is an edge i.e., $\dfrac{n}{r}=2$ and the two non-cut vertices are adjacent. Now $n=2r$ and $\vert L\vert=2(r-1)$. Let $V(B)=\{u_{1},\ldots,u_{r},v_{1},\ldots,v_{r}\}$ be such that $u_{1}$ and $v_{1}$ are the only non-cut vertices of $\overline{B}$ belonging to $B$. Then $\{u_{1},v_{1}\}\in E(B)$. Without loss of generality we can assume $\mathcal{N}_{B}(u_{1})=\{u_{2},\ldots,u_{r}\}$ and $\mathcal{N}_{B}(v_{1})=\{v_{2},\ldots,v_{r}\}$. Suppose $\{u_{i},u_{j}\}\not\in E(B)$ for some $2\leq i,j\leq r$. Now consider $T=\mathcal{N}_{B}(v_{1})\cup \mathcal{N}_{B}(u_{i})\cup \mathcal{N}_{B}(u_{j})$. Now $T$ contains $\vert T\vert -1$ cut vertices of $\overline{B}$ and $B\setminus T$ contains at least three connected components which include $\{u_{i}\},\{u_{j}\},\{v_{1}\}$. Then it is easy to see that $T\in \mathfrak{C}(\overline{B})$ and $c_{\overline{B}}(T)\geq \vert T\vert -1+3=\vert T\vert +2$, a contradiction to the fact that $J_{\overline{B}}$ is unmixed. Therefore $\mathcal{N}_{B}(u_{1})\setminus \{v_{1}\}$ is complete and similarly, $\mathcal{N}_{B}(v_{1})\setminus\{u_{1}\}$ is also complete. Since, $B$ is $r$-regular, the only possibility is $B=K_{r}\star K_{r}$ (doing some relabeling). Hence $\overline{B}=G_{r}=\overline{K_{r}\star K_{r}}$. 
\end{proof}

\begin{theorem}\label{starthm}
Let $G=\overline{B}$ be a connected graph such that $B$ is a non-complete $r$-regular $r$-connected. Then the following are equivalent:
\begin{enumerate}[(i)]
\item $G=\overline{K_{r}\star K_{r}}$ with $r\geq 2$;
\item $J_{G}$ is Cohen-Macaulay;
\item $R/J_{G}$ is $S_{2}$;
\item $G$ is accessible;
\item $J_{G}$ is strongly unmixed.
\end{enumerate}
\end{theorem}

\begin{proof}
$\text{(i)}\,\, \Rightarrow\,\, \text{(v)}$ follows from Theorem \ref{starsu}. We know 
$$\text{(v)}\,\, \Rightarrow \,\,\text{(ii)}\,\, \Rightarrow \,\,\text{(iii)}\,\, \Rightarrow\,\, \text{(iv)}.$$
 Also, by Proposition \ref{r-regacc}, $\text{(iv)}\,\, \Rightarrow\,\, \text{(i)}$. Therefore we have
$$\text{(i)}\,\,\Rightarrow\,\, \text{(v)}\,\, \Rightarrow \,\,\text{(ii)}\,\, \Rightarrow \,\,\text{(iii)}\,\, \Rightarrow\,\, \text{(iv)}\,\,\Rightarrow\,\,\text{(i)}.$$
\end{proof}

\begin{figure}[H]
	\centering
	\begin{tikzpicture}
  [scale=.6,auto=left,every node/.style={circle,scale=0.5}]
    
  \node[draw] (n1) at (0,4)  {$1$};
  \node[draw] (n2) at (0,0)  {$2$};
  \node[draw] (n3) at (2,2) {$3$};
   \node[draw] (m1) at (7,4) {$4$};
   \node[draw] (m2) at (7,0) {$5$};
  \node[draw] (m3) at (5,2) {$6$};
  \node[draw] (fn1) at (-3,4) {$7$};
  \node[draw] (fm1) at (10,4) {$8$};
  \node[draw] (fn2) at (-3,0) {$9$};
  \node[draw] (fm2) at (10,0) {$10$};

 \foreach \from/\to in {n1/n2,n1/n3, n1/fn1, n1/m1, n2/n3, n2/fn2, n2/m2, n3/m3, m1/m2, m1/m3, m2/m3, m2/fm2, m1/fm1}
    \draw[] (\from) -- (\to);
   
\end{tikzpicture}
	\caption{Planar graph $\overline{K_{3}\star K_{3}}$ with $J_{\overline{K_{3}\star K_{3}}}$ Cohen-Macaulay.}\label{figstar}
\end{figure}
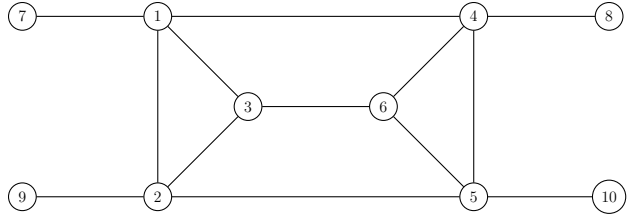

\begin{example}{\rm
In the \Cref{figstar}, the graph $\overline{K_{3}\star K_{3}}$ is a planar accessible graph such that $K_{3}\star K_{3}$ is $3$-regular $3$-connected. Note that for $r>3$, $K_{r}\star K_{r}$ is non-planar graph. By Theorem \ref{starthm}, $J_{\overline{K_{3}\star K_{3}}}$ is strongly unmixed (so is Cohen-Macaulay) and $K_{3}\star K_{3}$ is the largest $r$-regular $r$-connected non-complete planar graph with $\overline{K_{3}\star K_{3}}$ is accessible.
}
\end{example}

\section{Cohen-Macaulay Binomial Edge Ideals of Graphs Containing up to Three Cut Vertices }\label{seccutcon}
\medskip

In this section, we introduce a new family of graphs inductively. Then for these graphs having at most three cut vertices, we show that the equivalency of accessibility, Cohen-Macaulayness, and strongly unmixedness of binomial edge ideals holds.
\medskip

\begin{lemma}\label{cut3p1}
Let $B$ be a block of a graph such that $\overline{B}$ is accessible with three cut vertex $\{v_{1},v_{2},v_{3}\}$ and $\{v_{1},v_{2}\}, \{v_{2},v_{3}\}\in E(B)$ but $\{v_{1},v_{3}\}\not\in E(B)$. If there exists no path between $v_{1}$ and $v_{3}$ in $B\setminus \{v_{2},x_{1}\}$ and $B\setminus\{v_{2},y_{1}\}$, for some $x_{1}\in \mathcal{N}_{B}(v_{1})\setminus\{v_{2}\}$ and $y_{1}\in \mathcal{N}_{B}(v_{3})\setminus\{v_{2}\}$, then $J_{\overline{B}\setminus\{v_{i}\}}$ is unmixed for some $i\in\{1,2,3\}$.
\end{lemma}

\begin{proof}
$\overline{B}$ is accessible implies $J_{\overline{B}}$ is unmixed and so $c(\{v_{2}\})=2$. Therefore, in $B\setminus \{v_{2}\}$, there is a path between $v_{1}$ and $v_{3}$. By the given condition, $v_{1},x_{1},\ldots,y_{1},v_{3}$ is a path in $B\setminus \{v_{2}\}$. Suppose there exists no cut vertex $v_{i}$ of $\overline{B}$ such that $J_{\overline{B}\setminus\{v_{i}\}}$ is unmixed. Then by \cite[Proposition 6.1]{acc}, for each cut vertex $v_{i}\in V(B)$, where $1\leq i\leq 3$, there exists a cut set $S_{i}$ of $\overline{B}\setminus \{v_{i}\}$ containing $\mathcal{N}_{B}(v_{i})$ and so, by \cite[Remark 5.4]{acc}, $\mathcal{N}_{B}(v_{i})\in \mathfrak{C}(\overline{B})$ for all $1\leq i\leq 3$.\par

Let $\mathcal{N}_{B}(v_{1})=\{v_{2},x_{1},\ldots,x_{s}\}$ and assume $s\geq 2$. It is clear that $\{v_{2},x_{1}\}\in \mathfrak{C}(\overline{B})$ and also, we have $\mathcal{N}_{B}(v_{1})\in \mathfrak{C}(\overline{B})$. Then from Proposition \ref{accnoncut}, we can conclude $\{v_{2},x_{j}\}\in \mathfrak{C}(\overline{B})$ and $\{v_{2},x_{j},x_{i}\}\in \mathfrak{C}(\overline{B})$ for some $i,j\in\{1,\ldots,s\}$. Let $j\neq 1$. Then there is a connected component $A_{j}$ in $\overline{B}\setminus \{v_{2},x_{j}\}$ other than the component containing $\{v_{1},v_{3}\}$ and the component $\{f_{v_{2}}\}$. Also, $A_{j}$ is adjacent to only $x_{j}$ and $v_{2}$ in $\overline{B}$. Thus, Theorem \ref{accthm} implies that every vertex in $A_{j}$ is adjacent to the cut vertex $v_{2}$. Choose $y\in V(A_{j})$. In $(\overline{B}\setminus \{v_{2}\})\setminus (S_{2}\setminus \{y\})$, $y$ can not be a cut vertex and this contradicts the fact $\mathcal{N}_{B}(v_{2})\subseteq S_{2}\in \mathfrak{C}(\overline{B}\setminus\{v_{2}\})$. Therefore $j=1$ is the only possibility and without loss of generality taking $i=2$ we get $\{v_{2},x_{1},x_{2}\}\in \mathfrak{C}(\overline{B})$. As $J_{\overline{B}}$ is unmixed, we have $c_{\overline{B}}(\{v_{2},x_{1},x_{2}\})=4$. Let $A_{1},A_{2}, A_{3}, \{f_{v_{2}}\}$ be the four connected components of $\overline{B}\setminus \{v_{2},x_{1},x_{2}\}$ such that $v_{1}\in V(A_{1}), v_{3}\in V(A_{3})$. Now $A_{2}$ contains no cut vertex of $\overline{B}$ and vertices belong to $A_{2}$ can be adjacent to only $x_{1},x_{2}$ and $v_{2}$ in $\overline{B}$ outside $A_{2}$. Let $w\in V(A_{2})$ be any vertex. Then $w$ should be adjacent to a cut vertex in $\overline{B}$ as $\overline{B}$ is accessible and the only choice is $v_{2}$. Thus, $V(A_{2})\subseteq \mathcal{N}_{B}(v_{2})$. Again we have $\mathcal{N}_{B}(v_{2})\subseteq S_{2}\in \mathfrak{C}(\overline{B}\setminus\{v_{2}\})$. Therefore, in $(\overline{B}\setminus \{v_{2}\})\setminus (S_{2}\setminus \{w\})$, $w$ should be a cut vertex and so we should have $\{x_{1},w\}, \{x_{2},w\}\in E(B)$ and $\{x_{1},x_{2}\}\not\in E(B)$ with $x_{1},x_{2}\not\in \mathcal{N}_{B}(v_{2})$. These conditions hold for any vertex in $A_{2}$.\par 

\noindent\textbf{Case-I:} Let $v_{1},x_{1},z,\ldots,y_{1},v_{3}$ be a chordless path between $v_{1}$ and $v_{3}$ in $G\setminus \{v_{2}\}$, where $z, x_{1}, y_{1}$ are distinct. Then $w\neq z$ otherwise, $v_{1},x_{2},w=z,\ldots,y_{1},v_{3}$ would be a path. Since the path is chordless, $\{z,v_{2}\}\in E(B)$ follows from Theorem \ref{accthm}. Now $\mathcal{N}_{B}(v_{2})\in \mathfrak{C}(\overline{B})$ implies there exists $w^{\prime}\in \mathcal{N}_{B}(v_{2})$ such that $\{v_{1},v_{3},w^{\prime}\}\in \mathfrak{C}(\overline{B})$ by Proposition \ref{accnoncut}. Suppose $w^{\prime}=z$ and $\{f_{v_{1}}\},\{f_{v_{3}}\}, A^{z}_{v_{2}}, A^{z}$ are four connected components of $\overline{B}\setminus \{v_{1},v_{3},z\}$, where $v_{2}\in A^{z}_{v_{2}}$. Then any vertex in $A^{z}$ is adjacent to one of $v_{1}$ or $v_{3}$. As there is no path between $v_{1}$ and $v_{3}$ in $B\setminus \{v_{2},x_{1}\}$ and $B\setminus\{v_{2},y_{1}\}$, the only possibility is $A^{z}\subseteq\{x_{1},y_{1}\}$. But $\{x_{1},w\},\{w,v_{2}\}\in E(\overline{B})$ implies $x_{1}\not\in A^{z}$ and similarly, $y_{1}\not\in A^{z}$. Therefore $w^{\prime}\neq z$. Let $\{f_{v_{1}}\},\{f_{v_{3}}\}, A^{w^{\prime}}_{v_{2}}, A^{w^{\prime}}$ are four connected components of $\overline{B}\setminus \{v_{1},v_{3},w^{\prime}\}$, where $v_{2}\in A^{w^{\prime}}_{v_{2}}$. By Theorem \ref{accthm}, every vertex in $A^{w^{\prime}}$ is adjacent to $v_{1}$ or $v_{3}$. Note that $x_{1},y_{1}\in A^{w^{\prime}}_{v_{2}}$. If there exists $x^{\prime},y^{\prime}\in V(A^{w^{\prime}})$ such that $\{x^{\prime},v_{1}\}\in E(B)$ and $\{y^{\prime},v_{3}\}\in E(B)$, then $v_{1},x^{\prime},\ldots,y^{\prime},v_{3}$ or $v_{1},x^{\prime}=y^{\prime},v_{3}$ will be a path and it contradicts our given hypothesis. So $A^{w^{\prime}}$ is adjacent to exactly one of $v_{1}$ or $v_{3}$, without loss of generality, say $v_{1}$. Take $u \in V(A^{w^{\prime}})$. Now $\mathcal{N}_{B}(v_{1})\subseteq S_{1}\in \mathfrak{C}(\overline{B}\setminus\{v_{1}\})$ but in $(\overline{B}\setminus \{v_{1}\})\setminus (S_{1}\setminus \{u\})$, $u$ can not be a cut vertex and this gives a contradiction.\par 

\noindent\textbf{Case-II:} Let $v_{1},x_{1},y_{1},v_{3}$ be the only chordless path between $v_{1}$ and $v_{3}$ in $G\setminus \{v_{2}\}$. Then there exists $\{v_{2},x_{1},x_{2}\}\in \mathfrak{C}(\overline{B})$ and $\{v_{2},y_{1},y_{2}\}\in \mathfrak{C}(\overline{B})$, where $x_{1},x_{2}\in \mathcal{N}_{B}(v_{1})$, $y_{1},y_{2}\in \mathcal{N}_{B}(v_{3})$ such that $y_{1},y_{2}$ are similar as $x_{1},x_{2}$ as mentioned in the beginning. In this set up, it is easy to observe that, $S=\{v_{2},x_{1},x_{2},y_{1},y_{2}\}\in \mathfrak{C}(\overline{B})$. But $S\setminus \{s\}\not\in \mathfrak{C}(\overline{B})$ for any $s\in\{x_{1},x_{2},y_{1},y_{2}\}$ which contradicts the accessibility of $\overline{B}$.\par 

\noindent\textbf{Case-III:} Let $v_{1},x_{1}=y_{1},v_{3}$ be the path between $v_{1}$ and $v_{3}$ in $G\setminus \{v_{2}\}$. Now $\mathcal{N}_{B}(v_{2})\in \mathfrak{C}(\overline{B})$ implies there is $t\in \mathcal{N}_{B}(v_{2})$ such that $\{v_{1},v_{3},t\}\in \mathfrak{C}(\overline{B})$ by Proposition \ref{accnoncut}. Note that $\{v_{2},x_{1}\}\in \mathfrak{C}(\overline{B})$. Let $H_{1},H_{3},\{f_{v_{2}}\}$ be three connected components of $\overline{B}\setminus \{v_{2},x_{1}\}$ such that $v_{1}\in H_{1}$ and $v_{3}\in H_{3}$. We have $t\in H_{1}$ or $t\in H_{3}$, say $t\in H_{1}$. Suppose $\{f_{v_{1}}\}, \{f_{v_{3}}\}, A^{t}_{v_{2}}, A^{t}$ are the four connected components of $\overline{B}\setminus \{v_{1},v_{3},t\}$ such that $v_{2}\in A^{t}_{v_{2}}$. Observe that $x_{1}\in A^{t}_{v_{2}}$. Therefore, any $u\in A^{t}$ is adjacent to $v_{1}$, otherwise there will be a path between $v_{1}$ and $v_{3}$ in $\overline{B}\setminus \{v_{2},x_{1}\}$. As $u$ can be adjacent to only $v_{1}$ and $t$ other than the vertices of $V(A^{t})$, in $(\overline{B}\setminus \{v_{1}\})\setminus (S_{1}\setminus \{u\})$, $u$ can not be a cut vertex. Thus, $\mathcal{N}_{B}(v_{1})\subseteq S_{1}\in \mathfrak{C}(\overline{B})$ gives a contradiction.\par 

So our assumption $s\geq 2$ was wrong and we will examine the case $s=1$. Now $\mathcal{N}_{B}(v_{1})=\{v_{2},x_{1}\}$ and similarly, we have $\mathcal{N}_{B}(v_{3})=\{v_{2},y_{1}\}$. Then by \cite[Lemma 6.2]{acc}, we get $\{v_{2},x_{1}\}\not\in E(B)$ and $\{v_{2},y_{1}\}\not\in E(B)$.\par 

\noindent\textbf{Case-A:} Let $x_{1}\neq y_{1}$. Note that $\{v_{1},v_{3}\}\in \mathfrak{C}(\overline{B})$. So, $c(\{v_{1},v_{3}\})=3$ as $J_{\overline{B}}$ is unmixed. Therefore, $x_{1},y_{1},v_{2}$ are connected in $\overline{B}\setminus \{v_{1},v_{3}\}$. Also, note that any vertex in $V(B)\setminus \{x_{1},y_{1}\}$ is adjacent to $v_{2}$ by Theorem \ref{accthm}. So $ \mathcal{N}_{B}(v_{2})\setminus \{v_{1},v_{3}\}\neq \phi$. Since $\mathcal{N}_{B}(v_{2})\subseteq S_{2}\in \mathfrak{C}(\overline{B}\setminus \{v_{2}\})$, for any $w\in \mathcal{N}_{B}(v_{2})\setminus \{v_{1},v_{3}\}$, $w$ should be adjacent to both $x_{1}$, $y_{1}$ and also, $\{x_{1},y_{1}\}\not\in E(B)$. We have $\mathcal{N}_{B}(v_{2})\in \mathfrak{C}(\overline{B})$ and $\overline{B}\setminus \mathcal{N}_{B}(v_{2})$ has five connected components, namely, $\{f_{v_{1}}\},\{f_{v_{3}}\}, \overline{B}[\{f_{v_{2}}, v_{2}\}], \{x_{1}\}, \{y_{1}\}$. Thus, by unmixed property of $J_{\overline{B}}$ we have $\mathcal{N}_{B}(v_{2})=\{v_{1},v_{3}, w_{1},w_{2}\}$. But, we do not have $\{v_{1},v_{3},w_{1}\}$ or $\{v_{1},v_{3},w_{2}\}$ as a cut set of $\overline{B}$ which gives a contradiction by Proposition \ref{accnoncut} to the fact that $\overline{B}$ is accessible.\par 

\noindent\textbf{Case-B:} Let $x_{1}=y_{1}$. Again $\{v_{1},v_{3}\}\in \mathfrak{C}(\overline{B})$ and So, $c(\{v_{1},v_{3}\})=3$. Then there exists $w\in \mathcal{N}_{B}(v_{2})\setminus \{v_{1},v_{3}\}$ such that $\{w,x_{1}\}\in E(B)$. Now it is easy to observe that $\{v_{2},x_{1}\}\in \mathfrak{C}(\overline{B})$ but, $c(\{v_{2},x_{1}\})>3$ which is a contradiction.\\
Hence, our initial assumption was wrong and we can conclude that $J_{\overline{B}\setminus\{v_{i}\}}$ is unmixed for some $i\in\{1,2,3\}$.
\end{proof}

\begin{lemma}\label{cut3p2}
Let $B$ be a block of a graph such that $\overline{B}$ is accessible with three cut vertex $\{v_{1},v_{2},v_{3}\}$ and $\{v_{1},v_{2}\}, \{v_{2},v_{3}\}\in E(B)$ but $\{v_{1},v_{3}\}\not\in E(B)$. If there exists a path between $v_{1}$ and $v_{3}$ in $B\setminus \{v_{2},x\}$ for each $x\in \mathcal{N}_{B}(v_{1})$, then $J_{\overline{B}\setminus\{v_{i}\}}$ is unmixed for $i\in\{1,2,3\}$.
\end{lemma}

\begin{proof}
Suppose $J_{\overline{B}\setminus \{v_{i}\}}$ is not unmixed for all $1\leq i\leq 3$. Then by \cite[Proposition 6.1]{acc}, for each cut vertex $v_{i}\in V(B)$, $1\leq i\leq 3$, there is a cut set $S_{i}\in \mathfrak{C}(\overline{B}\setminus \{v_{i}\})$ such that $\mathcal{N}_{B}(v_{i})\subseteq S_{i}$ and so, by \cite[Remark 5.4]{acc}, $\mathcal{N}_{B}(v_{i})\in \mathfrak{C}(\overline{B})$ for all $1\leq i\leq 3$. As $\overline{B}$ is accessible by Proposition \ref{accnoncut}, $\mathcal{N}_{B}(v_{1})\in \mathfrak{C}(\overline{B})$ implies there exists $x_{1}\in \mathcal{N}_{B}(v_{1})$ such that $\{v_{2},x_{1}\}\in \mathfrak{C}(\overline{B})$. Since $J_{\overline{B}}$ is unmixed, $c(\{v_{2},x_{1}\})=3$. After removing $v_{2}$ and $x_{1}$, there will be a path between $v_{1}$ and $v_{3}$. Let $\{f_{v_{2}}\}, A^{x_{1}}, A_{v_{1}v_{3}}$ be the three components of $\overline{B}\setminus \{v_{2},x_{1}\}$ such that $v_{1},v_{3}\in A_{v_{1}v_{3}}$. By Theorem \ref{accthm}, any vertex of $A^{x_{1}}$ is adjacent to $v_{2}$. Also, any vertex of $A^{x_{1}}$ can be adjacent to only $x_{1}$ and $v_{2}$ outside $A^{x_{1}}$. Choose a vertex $w\in A^{x_{1}}$. Then $w$ can not be a cut vertex in $(\overline{B} \setminus \{v_{2}\})\setminus (S_{2}\setminus \{w\})$ and this contradicts the fact that $S_{2}\in \mathfrak{C}(\overline{B}\setminus \{v_{2}\})$. Hence, our assumption was wrong and $J_{\overline{B}\setminus\{v_{i}\}}$ is unmixed for $i\in\{1,2,3\}$.
\end{proof}

\begin{proposition}\label{3cutacc}
Let $B$ be a block such that $\overline{B}$ is accessible with three cut vertices. Then there exists a cut vertex $v$ of $\overline{B}$ for which $J_{\overline{B}\setminus\{v\}}$ is unmixed.
\end{proposition}

\begin{proof}
If the induced subgraph on the three cut vertex of $\overline{B}$ is complete then by \cite[Proposition 6.6]{acc}, there is a cut vertex $v\in V(B)$ of $\overline{B}$ such that $J_{\overline{B}\setminus \{v\}}$ is unmixed. Let $v_{1},v_{2},v_{3}$ be the cut vertices belonging to $V(B)$ and assume $B[\{v_{1},v_{2},v_{3}\}]$ is not complete. By Theorem \ref{accthm}, $B[\{v_{1},v_{2},v_{3}\}]$ is connected and so without loss of generality let $E(B[\{v_{1},v_{2},v_{3}\}])=\{\{v_{1},v_{2}\}, \{v_{2},v_{3}\}\}$. \par

Note that $\{v_{2}\}\in \mathfrak{C}(\overline{B})$ and thus, in $\overline{B}\setminus \{v_{2}\}$ there will be a path between $v_{1}$ and $v_{3}$ as $J_{\overline{B}}$ is unmixed. Suppose there exists vertices $x\in \mathcal{N}_{B}(v_{1})\setminus \{v_{2}\}$ and $y\in \mathcal{N}_{B}(v_{3})\setminus \{v_{2}\}$ such that there is no path between $v_{1}$ and $v_{3}$ in $\overline{B}\setminus \{v_{2},x\}$ and also, in $\overline{B}\setminus \{v_{2},y\}$. Then by Lemma \ref{cut3p1}, $J_{\overline{B}\setminus\{v_{i}\}}$ is unmixed for some $i\in\{1,2,3\}$. Now assume for each vertex $x\in \mathcal{N}_{B}(v_{1})$ there is a path between $v_{1}$ and $v_{3}$ in $\overline{B}\setminus \{v_{1},x\}$. Then by Lemma \ref{cut3p2}, $J_{\overline{B}\setminus\{v_{i}\}}$ is unmixed for some $i\in\{1,2,3\}$. Similarly, if there is a path between $v_{1}$ and $v_{3}$ in $\overline{B}\setminus \{v_{3},y\}$ for each vertex $y\in \mathcal{N}_{B}(v_{3})$, then again by Lemma \ref{cut3p2}, $J_{\overline{B}\setminus\{v_{i}\}}$ is unmixed for some $i\in\{1,2,3\}$.
\end{proof}

\begin{definition}\label{r-cutcon}{\rm
A connected graph $G$ is said to be $r$-\textit{cut-connected} if $G$ has no cut vertex or for any cut vertex $v$ of $G$, the number of cut vertices in any connected component of $G\setminus\{v\}$ is less than or equal to $r$.\par 
For a disconnected graph $G$, if every connected components of $G$ is $r$-cut-connected, then we call $G$  is $r$-cut-connected
}
\end{definition}

\begin{definition}\label{strongly-r-cutcon}{\rm
A graph $G$ is called \textit{strongly $r$-cut-connected} if $G$ is $r$-cut-connected and for any cut vertex $v$ of $G$, $G\setminus\{v\}$ is strongly $r$-cut-connected.
}
\end{definition}

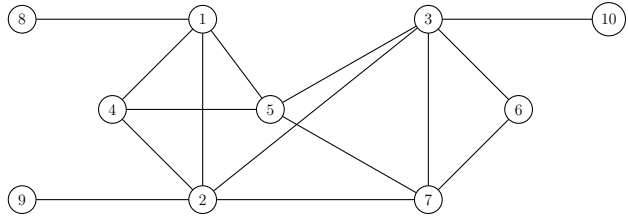
\begin{figure}[H]
\begin{center}
\begin{tikzpicture}
  [scale=.6,auto=left,every node/.style={circle,scale=0.5}]
 
  \node[draw] (n1) at (0,4)  {$1$};
  \node[draw] (n2) at (0,0)  {$2$};
  \node[draw] (n3) at (5,4) {$3$};
   \node[draw] (n4) at (-2,2) {$4$};
   \node[draw] (n5) at (1.5,2) {$5$};
  \node[draw] (n6) at (7,2) {$6$};
  \node[draw] (n7) at (5,0) {$7$};
  \node[draw] (n8) at (-4,4) {$8$};
  \node[draw] (n9) at (-4,0) {$9$};
  \node[draw] (n10) at (9,4) {$10$};
 
  \foreach \from/\to in {n1/n2,n1/n5,n1/n4,n1/n8, n2/n3, n2/n4, n2/n7, n2/n9, n3/n5, n3/n6, n3/n7,n3/n10, n4/n5, n5/n7, n6/n7}
    \draw[] (\from) -- (\to);
   
\end{tikzpicture}
\end{center} 
\caption{Strongly $3$-cut-connected graph $G$ with $3$ cut vertices and $J_{G}$ Cohen-Macaulay.}\label{fig3cutcon}
\end{figure}

\begin{lemma}\label{cutcon}
If $G$ is strongly $r$-cut-connected, then for any cut vertex $v$ of $G$, $G_{v}$ and $G_{v}\setminus\{v\}$ are strongly $r$-cut-connected. 
\end{lemma}

\begin{proof}
It is clear that $G$ is $r$-cut-connected implies $G_{v}$ and $G_{v}\setminus\{v\}$ are also $r$-cut-connected. We use induction on the number of vertices of $G_{v}$. If $G_{v}$ has no cut vertex, then we are done and note that for the base case $G_{v}$ has no cut vertex. Let $u$ be a cut vertex of $G_{v}$. Then $u$ is also a cut vertex of $G$. Now $G\setminus\{u\}$ is strongly $r$-cut-connected and so is $(G\setminus\{u\})_{v}$ by induction hypothesis. Note that $(G\setminus\{u\})_{v}= G_{v}\setminus\{u\}$ and hence $G_{v}$ is strongly $r$-cut-connected. In a similar way, $G_{v}\setminus\{v\}$ is also strongly $r$-cut-connected.
\end{proof}

\begin{theorem}\label{cutconthm}
Let $G$ be a strongly $3$-cut-connected graph having at most three cut vertices in any connected component. Then the following properties of $G$ are equivalent:
\begin{enumerate}[(i)]
\item $J_{G}$ is Cohen-Macaulay;
\item $R/J_{G}$ is $S_{2}$;
\item $G$ is accessible;
\item $J_{G}$ is strongly unmixed.
\end{enumerate}
\end{theorem}

\begin{proof}
(iv) $\Rightarrow$ (i) $\Rightarrow$ (ii) $\Rightarrow$ (iii) is known.\par
(iii) $\Rightarrow$ (iv): Let $\mathcal{G}$ be the class of strongly $3$-cut-connected accessible graphs having at most three cut vertices in any connected component. Let $G\in\mathcal{G}$. If for every block $B$ of $G$, the induced subgraph on the cut vertices of $G$ belong to $V(B)$ is complete, then by \cite[Proposition 6.6]{acc}, there exists a cut vertex $v$ of $G$ for which $G\setminus\{v\}$ is unmixed. Now assume the induced subgraph on the cut vertices of $G$ belong to a block is not complete. Since $G$ is accessible, by Theorem \ref{accthm}, the only possibility is there is a block of $G$ containing all three cut vertices and the induced subgraph on the cut vertices is a $P_{2}$. In this case, using Proposition \ref{3cutacc}, we get a cut vertex $v$ such that $G\setminus\{v\}$ is unmixed. Therefore, by \cite[Corollary 5.16]{acc}, $G\setminus\{v\}, G_{v}, G_{v}\setminus\{v\}$ are accessible. By definition $G\setminus\{v\}$ is strongly $3$-cut-connected and by Lemma \ref{cutcon}, $G_{v}, G_{v}\setminus\{v\}$ are strongly $3$-cut-connected. Since $G$ is $3$-cut-connected, number of cut vertices in any connected components of $G\setminus\{v\}, G_{v}$ and $G_{v}\setminus\{v\}$ are less than or equal to three. Thus, $G\setminus\{v\}, G_{v}, G_{v}\setminus\{v\}\in \mathcal{G}$ and hence by \cite[Proposition 5.13]{acc}, $J_{G}$ is strongly unmixed.
\end{proof}

\begin{example}{\rm
Consider the graph $G$ in \Cref{fig3cutcon}. By computing the primary decomposition of $J_{G}$ using Singular, we get 
\begin{align*}
\mathfrak{C}(G)=&\{\phi, \{1\},\{2\},\{3\}, \{1,2\},\{1,3\}, \{2,3\},\{2,5\}, \{3,7\},\\ &\{1,2,3\}\{1,2,5\}, \{1,3,7\}, \{2,3,5\}, \{2,3,7\},\\ &\{1,2,3,5\}, \{1,2,3,7\}, \{1,3,4,7\}\}.
\end{align*}
Note that $G$ is accessible. Now $G\setminus\{1\}$ contains two cut vertices $2$ and $3$. The sets of cut vertices in $G\setminus\{2\}$ and $G\setminus\{3\}$ are $\{1,3,5\}$ and $\{1,2,7\}$, respectively. Again, the sets of cut vertices of $G\setminus\{1,2\}, G\setminus\{1,3\}, G\setminus\{2,3\}, G\setminus\{2,5\}, G\setminus\{3,7\}$ are $\{3,5\}, \{2,7\}, \{1,5,7\}, \{1,3\}, \{1,2\}$, respectively. Also, $G\setminus\{1,2,3\}, G\setminus\{1,2,5\}, G\setminus\{1,3,7\}, G\setminus\{2,3,5\}, G\setminus\{2,3,7\}$ contains $\{5,7\}, \{3\},\{2,4\}, \{1\},\{1\}$, respectively as a set of cut vertices. If we go further, then there will be no cut vertices left. From this observation, it is clear that $G$ is strongly $3$-cut-connected graph with three cut vertices. Hence by Theorem \ref{cutconthm}, $J_{G}$ is strongly unmixed and so is Cohen-Macaulay.
}
\end{example}

\begin{theorem}\label{allblocksu}
Let $G$ be a graph such that every block $B$ of $G$ satisfies any of the following conditions:\\
$(a)$ $B$ is chordal; $(b)$ $\overline{B}$ is traceable; $(c)$ $B$ is a chain of cycles (see \cite[Definition 4.2]{s2acc}; $(d)$ $B=K_{m}\star_{r} K_{n}$; $(e)$  $\overline{B}$ is strongly $3$-cut-connected containing at most $3$ cut vertices of $G$. Then the following are equivalent:
\begin{enumerate}[(i)]
\item $J_{G}$ is Cohen-Macaulay;
\item $R/J_{G}$ is $S_{2}$;
\item $G$ is accessible;
\item $J_{G}$ is unmixed and each $\overline{B}$ is accessible.
\item $J_{G}$ is strongly unmixed.
\end{enumerate}
\end{theorem}

\begin{proof}
(v) $\Rightarrow$ (i) $\Rightarrow$ (ii) $\Rightarrow$ (iii) is clear.\\ 
(iii) $\Leftrightarrow$ (iv): Follows from Theorem \ref{baracc} and Theorem \ref{baraccG}.\\
(iv) $\Rightarrow$ (v): If $B$ is chordal, then $J_{\overline{B}}$ is strongly unmixed by \cite[Theorem 6.4]{acc} and if $\overline{B}$ is traceable. then by \cite[Theorem 6.8]{acc}, $J_{\overline{B}}$ is strongly unmixed. For a block $B$ which is a block of chains, strongly unmixed property of $J_{\overline{B}}$ follows from \cite[Theorem 4.17]{s2acc}. By Lemma \ref{staracc}, Theorem \ref{starsu} and the structure of $K_{m}\star_{r} K_{n}$ it follows that $J_{\overline{K_{m}\star_{r} K_{n}}}$ is strongly unmixed if and only if $\overline{K_{m}\star_{r} K_{n}}$ is accessible. If $B$ belongs to the category $(e)$, then $J_{\overline{B}}$ is strongly unmixed by Theorem \ref{cutconthm}. Now $J_{G}$ is unmixed and we see for each block $B$ of $G$, $J_{\overline{B}}$ is strongly unmixed. Hence by Theorem \ref{barsuG}, $J_{G}$ is strongly unmixed.
\end{proof}

\begin{example}\label{exallblocksu}{\rm
Consider the graph $G$ in \Cref{figbarsuG}. We see $G$ has $5$ blocks $B_{1},\ldots, B_{5}$ and for each $B_{i}$ we consider the graph $\overline{B_{i}}$ with respect to $G$. Using Singular (\cite{singular}), we check that each $\overline{B_{i}}$ is accessible. Therefore, by \cite[Theorem 6.8]{acc}, $J_{\overline{B_{1}}}$ is strongly unmixed and by \cite[Theorem 6.4]{acc}, $J_{\overline{B_{2}}}$ is strongly unmixed. We proved in Example $J_{\overline{B_{3}}}$ is strongly unmixed. Also, strongly unmixed property of $J_{\overline{B_{4}}}$ follows from \cite[Theorem 4.17]{s2acc} and $J_{\overline{B_{5}}}$ is strongly unmixed by Theorem \ref{starthm}. Using repeating application of Corollary \ref{unmixed} we see that $J_{G}$ is unmixed. Thus, by Theorem \ref{barsuG}, $J_{G}$ is strongly unmixed. (or by Theorem \ref{allblocksu}.)
}
\end{example}

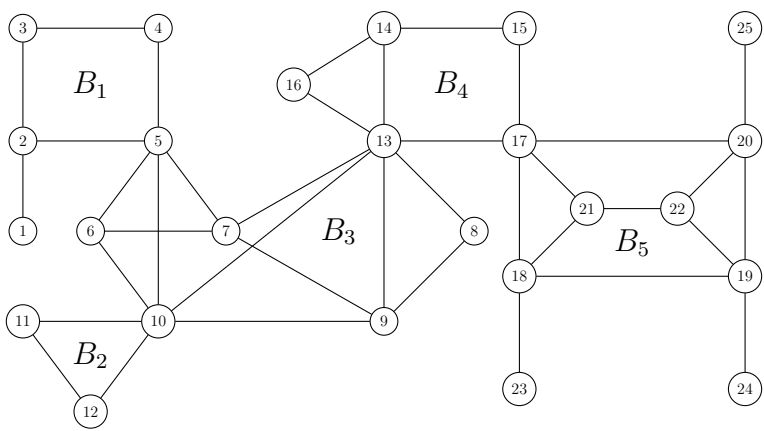
\begin{figure}[H]
\begin{center}
\begin{tikzpicture}
  [scale=.6,auto=left,every node/.style={circle,scale=0.5}]
  \node[draw] (n1) at (-4,2)  {$1$};
  \node[draw] (n2) at (-4,4)  {$2$};
  \node[draw] (n3) at (-4,6.5)  {$3$};
  \node[draw] (n4) at (-1,6.5)  {$4$};
  \node[draw] (n5) at (-1,4)  {$5$};
  \node[draw] (n6) at (-2.5,2)  {$6$};
  \node[draw] (n7) at (0.5,2)  {$7$};
  \node[draw] (n8) at (6,2) {$8$};
  \node[draw] (n9) at (4,0) {$9$};
  \node[draw] (n10) at (-1,0) {$10$};
  \node[draw] (n11) at (-4,0) {$11$};
  \node[draw] (n12) at (-2.5,-2) {$12$};
  \node[draw] (n13) at (4,4) {$13$};
  \node[draw] (n14) at (4,6.5) {$14$};
  \node[draw] (n15) at (7,6.5) {$15$};
  \node[draw] (n16) at (2,5.25) {$16$};
  \node[draw] (n17) at (7,4)  {$17$};
  \node[draw] (n18) at (7,1)  {$18$};
  \node[draw] (n19) at (12,1)  {$19$};
  \node[draw] (n20) at (12,4)  {$20$};
  \node[draw] (n21) at (8.5,2.5)  {$21$};
  \node[draw] (n22) at (10.5,2.5)  {$22$};
  \node[draw] (n23) at (7,-1.5)  {$23$};
  \node[draw] (n24) at (12,-1.5)  {$24$};
  \node[draw] (n25) at (12,6.5)  {$25$};
\node[scale=2] (B1) at (-2.5,5.25) {$B_{1}$};
\node[scale=2] (B2) at (-2.5,-0.75) {$B_{2}$};
\node[scale=2] (B3) at (3,2) {$B_{3}$};
\node[scale=2] (B4) at (5.5,5.25) {$B_{4}$};
\node[scale=2] (B5) at (9.5,1.75) {$B_{5}$};
 
  \foreach \from/\to in {n1/n2,n2/n3,n2/n5,n3/n4, n4/n5, n5/n6, n5/n7, n5/n10, n6/n10, n6/n7, n7/n9, n7/n13, n8/n9, n8/n13, n9/n10, n9/n13, n10/n11, n10/n12, n10/n13, n11/n12, n13/n14, n13/n17, n14/n15, n15/n17, n13/n16, n14/n16, n17/n18, n17/n20, n17/n21, n18/n19, n18/n21, n18/n23, n19/n20, n19/n22, n19/n24, n20/n22, n20/n25, n21/n22}
    \draw[] (\from) -- (\to);
   
\end{tikzpicture}
\end{center} 
\caption{Graph $G$ with $J_{G}$ and each $J_{\overline{B}}$ Cohen-Macaulay (Moreover, strongly unmixed) for each block $B$ of $G$.}\label{figbarsuG}

\end{figure}

As a consequence of our Theorem \ref{baracc}, \ref{blocksu}, \ref{baraccG} and \ref{barsuG}, we are proposing the following Question \ref{quescmblock}.

\begin{question}\label{quescmblock}{\rm
For a connected graph $G$, is it true that $J_{G}$ is Cohen-Macaulay if and only if $J_{G}$ is unmixed and $J_{\overline{B}}$ is Cohen-Macaulay for each block $B$ of $G?$
}
\end{question}

We answer the Question \ref{quesidentify} arise in \cite{acc}, for the case of accessibility and strongly unmixedness in our Theorem \ref{identifythm}. So we are repeating the Question \ref{quesidentify} with the unproven part as follows.

\begin{question}\label{quesidentifycm}{\rm
Let $G$ and $H$ be two disjoint connected graphs such that $J_{G}$ and $J_{H}$ are unmixed. Let $v,w$ be the cut vertices of $G, H$, respectively, for which $J_{G\setminus\{v\}}$ and $J_{H\setminus\{w\}}$ are unmixed. Set $G\setminus\{v\} = G_{1}\sqcup G_{2}, H\setminus\{w\} = H_{1}\sqcup H_{2}$. Let $F_{ij}$ be the graph obtained by gluing $G[V(G_{i})\cup \{v\}]$ and $H[V(H_{j})\cup \{w\}]$ identifying $v$ and $w$, where $i,j = 1,2$. If $J_{G}$ and $J_{H}$ are Cohen-Macaulay, is it true that $J_{F_{ij}}$ is Cohen-Macaulay$?$
}
\end{question}

\bibliographystyle{amsalpha}

\end{document}